\documentclass[a4paper,10pt]{article}

\usepackage{a4wide}
\usepackage[utf8]{inputenc}
\usepackage{amsfonts,amsmath,amsthm,amssymb}
\usepackage{graphicx, color}

\newtheorem{lemma}{Lemma}[section]

\newtheorem{definition}{Definition}[section]
\newtheorem{theorem}{Theorem}[section]
\newtheorem{proposition}{Proposition}[section]
\newtheorem{example}{Example}[section]
\newtheorem{remark}{Remark}[section]
\newtheorem{corollary}{Corollary}[section]

\newcommand{\I}{{\rm Id}}
\newcommand{\bd}{{\rm bd\,}}
\newcommand{\skalp}[1]{\langle #1\rangle}
\newcommand{\B}{{\mathbb{B}}}
\newcommand{\dist}[2]{{\rm d}_{#1}(#2)}
\newcommand{\proj}[2]{{\rm P}_{#1}(#2)}
\newcommand{\argmin}{{\rm argmin\,}}
\newcommand{\epi}{{\rm epi\,}}
\newcommand{\Gr}{{\rm gph\,}}
\newcommand{\dom}{{\rm dom\,}}
\newcommand{\xb}{\bar x}
\newcommand{\yb}{\bar y}
\newcommand{\norm}[1]{\Vert #1 \Vert}
\newcommand{\mv}{\,\vert\,}
\newcommand{\R}{\mathbb{R}}
\newcommand{\N}{\mathbb{N}}

\newcommand{\setto}[1]{\mathop \to\limits^{#1}}

\newcommand{\co}{{\rm conv\,}}

\newcommand{\email}[1]{{\tt #1}}

\title{Calculus for directional limiting normal cones and subdifferentials}
\author{Mat\'{u}\v{s} Benko\thanks{
	      Institute of Computational Mathematics, Johannes Kepler University Linz,
              A-4040 Linz, Austria, \email{benko@numa.uni-linz.ac.at}},
              Helmut Gfrerer\thanks{
              Institute of Computational Mathematics, Johannes Kepler University Linz,
              A-4040 Linz, Austria, \email{helmut.gfrerer@jku.at}},
              Ji\v{r}\'{i} V. Outrata\thanks{
Czech Academy of Sciences, Institute of Information Theory and Automation,  
18208 Prague, Czech Republic, and Centre for Informatics and Applied Optimization, Federation
University of Australia, Ballarat, Vic 3350, Australia, \email{outrata@utia.cas.cz}}}

\date{}

\begin{document}

\maketitle

\begin{abstract}
The paper is devoted to the development of a comprehensive calculus for directional limiting normal cones, subdifferentials
and coderivatives in finite dimensions. This calculus encompasses the whole range of the standard generalized differential
calculus for (non-directional) limiting notions and relies on very weak (non-restrictive) qualification conditions
having also a directional character. The derived rules facilitate the application of tools exploiting
the directional limiting notions to difficult problems of variational analysis including, for instance,
various stability and sensitivity issues. This is illustrated by some selected applications in the last part of the paper.
\end{abstract}

\section{Introduction}

Since the early works of Mordukhovich, the limiting normal cone and the corresponding subdifferential
belong to the central notions of variational analysis. They admit a rich calculus both in finite
as well as in infinite dimensions and have been successfully utilized in a large variety of optimization and equilibrium problems,
see \cite{RoWe98}, \cite{Mo06a} and the references therein.
In particular, these notions play an important role in stability and sensitivity issues, above all
in analysis of various Lipschitzian stability notions related to multifunctions.
Having been motivated by some of the above listed applications,
in \cite{GM} the authors refined the original definitions
by restricting the limiting process only to a subset of sequences used in the original definitions.
This lead eventually to the notions of directional limiting normal cone and directional limiting subdifferential
which have been further developed and utilized above all in the works authored or coauthored by Gfrerer
\cite{Gfr13a,Gfr13b,Gfr14a,Gfr14b,GfrKl16,GO3,GO4,GY}.
It turned out that these directional notions (together with the directional limiting coderivative) enable us
indeed a substantially finer analysis of situations in which the estimates, provided by the standard calculus,
are too rough and so the corresponding assertions are not very useful.
This is, e.g., the case of verification  of metric subregularity of {\em feasibility mappings}
(calmness of {\em perturbation mappings}) related to constraint systems,
which lead to new first- and second-order sufficient conditions for metric subregularity \cite{Gfr13a,Gfr13b}.
They are now widely used as weak (non-restrictive) but yet verifiable qualification conditions \cite{GfrKl16,GY,YZ}.
In \cite{GO3} the authors used the directional limiting coderivative to establish new weak conditions
ensuring the calmness and the Aubin property of rather general implicitly defined multifunctions and in \cite{GO4}
this technique has then been worked out for a class of parameterized variational systems.

Further, directional limiting coderivative appears in sharp first-order optimality conditions \cite{Gfr13a},
entitled extended M-stationarity in \cite{Gfr14a},
which provide a dual characterization of B-stationarity for disjunctive programs.

One can definitely imagine also numerous other problems of variational analysis in which
the directional notions could be successfully employed. In all of them, however,
one needs a set of rules enabling us to compute efficiently the directional normal cones
and subdifferentials of concrete sets and functions in a similar way
like in the standard generalized differential calculus of Mordukhovich.
Some parts of such a calculus have already been conducted in connection with various results mentioned above.
In particular, in \cite{Gfr13b} one finds, apart from some elementary rules,
formulas for directional limiting normal cones to unions of convex polyhedral sets
and in \cite{GO3}, \cite{GO4} second-order chain rules have been derived
which enable us to compute directional limiting coderivatives to normal-cone mappings associated with various types of sets.
Further, \cite{LoWaYa17} contains several rules of the directional calculus even in Banach spaces,
in \cite{Viet} the authors proved a special coderivative sum rule
and in \cite{YZ} the situation has been examined when one has to do with the Carthesian products of sets and mappings.

The aim of this paper is to fill in this gap by building a systematic decent calculus of directional limiting normal cones
and subdifferentials following essentially the lines of \cite[Chapter 3]{Mo06a} and \cite{IoOut08}. The structure is as follows.

In Section 2 we collect the needed definitions and present some auxiliary results used throughout the paper.
Section 3 is devoted to the calculus of directional limiting normal cones.
As the most important results we consider formulas for the directional limiting normal cone of the pre-image
and of the image of a set in a Lipschitz continuous mapping. These results have numerous important consequences.
Section 4 concerns the calculus of directional limiting subdifferentials.
Apart from the chain and sum rules we consider the case of value functions,
distance functions, pointwise minima and maxima and examine also the partial directional limiting subdifferentials.
Section 5 provides formulas for directional limiting coderivatives of compositions
and sums of multifunctions together with some important special cases.
In Section 6 we present some problems of variational analysis, where the usage of the directional limiting calculus leads
to weaker (less restrictive) sufficient conditions or sharper (more precise) estimates.

Similarly as in \cite{IoOut08}, we have attempted to impose the ``weakest'' qualification conditions
expressed mostly in terms of directional metric subregularity of associated {\em feasibility mappings}.
Admittedly, these conditions are not always verifiable, but they can in many cases be replaced by stronger
(more restrictive) conditions expressed in terms of problem data.

In the directional calculus one also meets a new specific difficulty associated with the computation of the images or the pre-images
of the given {\em direction} in the considered mappings.
This obstacle leads in some cases to more complicated rules or to additional
qualification conditions.

It turns out, that the essential role in the calculus is played by Theorems \ref{The : ConstSet} and \ref{The : ConstSetForw},
concerning the pre-image and the image of a set in a continuous/Lipschitz mapping.
The basic ideas arising in these two statements appear in fact in almost all calculus rules throughout the whole paper.

The following notation is employed. The closed unit ball and the unit sphere in $\R^n$
are denoted by $\mathbb{B}$ and $\mathbb{S}$, respectively, while $\mathbb{B}_r(\xb):=\{x \in \R^n \mv \norm{x-\xb} \leq r\}$.
The identity mapping is denoted by $\I$.
Given a set $\Omega \subset \R^n$, $\bd \Omega$ stands for the boundary of $\Omega$, i.e., the set of points whose every neighborhood
contains a point of $\Omega$ as well as a point not belonging to $\Omega$. Moreover, given also a point $\xb$,
$\dist{\Omega}{\xb}$ denotes the distance from $\xb$ to set $\Omega$ and $\proj{\Omega}{\xb}$
denotes the projection of $\xb$ onto $\Omega$. For a sequence $x_k$, $x_k \setto{\Omega} \xb$
stands for $x_k \to \xb$ with $x_k \in \Omega$. Given a directionally differentiable function $\varphi: \R^n \to \R^m$
at $\xb \in \dom \varphi$, $\varphi^{\prime}(\bar{x};h)$ denotes the directional derivative of $\varphi$ at $\xb$ in direction $h$.
Finally, following traditional patterns, we denote by $o(t)$ for $t \geq 0$ a term with the property that $o(t)/t \to 0$
when $t \downarrow 0$.

\section{Preliminaries}

We start by recalling several definitions and results from variational analysis.
Let $\Omega\subset\R^n$ be an arbitrary closed set and $\xb\in\Omega$.
The {\em contingent} (also called  {\em Bouligand} or {\em tangent}) {\em cone} to $\Omega$ at $\xb$,
denoted by $T_{\Omega}(\xb)$, is given by
\[T_{\Omega}(\xb):=\{u\in \R^n\mv \exists (u_k)\to u, (t_k)\downarrow 0: \xb+t_ku_k\in\Omega \ \forall k\}.\]
We denote by
\begin{equation}\label{DefEpsNormals}
\hat N_{\Omega}(\xb):=\{\xi \in\R^n\mv \limsup_{x'\setto \Omega \xb}\frac{\xi^T( x'-\xb)}{\norm{x'-\xb}}\leq 0\}
\end{equation}
the {\em Fr\'echet} ({\em regular}) {\em normal cone} to $\Omega$ at $\xb$.
Finally, the {\em limiting} ({\em Mordukhovich}) {\em normal cone} to $\Omega$ at $\xb$ is
defined by
\[N_{\Omega}(\xb):=\{\xi \in\R^n \mv \exists (x_k)\setto{\Omega}\xb,\ (\xi_k)\to \xi: \xi_k\in \hat N_{\Omega}(x_k) \ \forall k\}.\]
If $\xb \notin \Omega$ we put $T_{\Omega}(\xb)=\emptyset$, $\hat N_{\Omega}(\xb)=\emptyset$ and $N_{\Omega}(\xb)=\emptyset$.

The Mordukhovich normal cone is generally nonconvex whereas the
Fr\'echet normal cone is always convex. In the case of a convex
set $\Omega$, both the Fr\'echet normal cone and the Mordukhovich normal cone
coincide with the standard normal cone from convex analysis and,
moreover, the contingent cone is equal to the tangent cone in the
sense of convex analysis.

Note that $\xi\in\hat N_{\Omega}(\xb)$ $\Leftrightarrow$ $\xi^Tu \leq 0\ \forall u\in T_{\Omega}(\xb)$,
i.e., $\hat N_{\Omega}(\xb)= \hat N_{T_{\Omega}(\xb)}(0)=T_{\Omega}(\xb)^\circ$ is the polar cone of $T_{\Omega}(\xb)$.

Consider an extended real-valued function $f: \R^n \to \bar \R$ and a point
$\xb \in \dom f := \{x \in \R^n \mv f(x) \in \R \}$, where $\dom f$ denotes the domain of $f$.
The {\em Fr\'echet} ({\em regular}) {\em subdifferential} of $f$ at $\xb$ is a set
$\hat \partial f(\xb)$ consisting of all $\xi \in \R^n$ such that
\begin{equation*}\label{DefRegSubdif}
 f(x) \geq f(\xb) + \skalp{\xi,x-\xb} + o(\norm{x-\xb}),
\end{equation*}
the {\em limiting} ({\em Mordukhovich}) {\em subdifferential} of $f$ at $\xb$, denoted by $\partial f(\xb)$, is given by
\[
 \partial f(\xb) := \{\xi \in \R^n \mv \exists (x_k) \setto{f} \xb,\ (\xi_k)\to \xi: \xi_k \in \hat \partial f(x_k) \ \forall k \}
\]
and the {\em singular subdifferential} of $f$ at $\xb$ is defined by
\[
 \partial^{\infty} f(\xb) := \{\xi \in \R^n \mv \exists (\lambda_k) \downarrow 0, (x_k) \setto{f} \xb,
 \xi_k \in \hat \partial f(x_k) \ \forall k: (\lambda_k \xi_k) \to \xi\}.
\]

Denoting by $\epi f := \{(x,\alpha) \in R^{n+1} \mv a \geq f(x)\}$ the epigraph of $f$, there is a well-known
equivalent description of the subdifferentials, namely
\begin{eqnarray*}
 \hat \partial f(\xb) & = & \{\xi \in \R^n \mv (\xi,-1) \in \hat N_{\epi f}(\xb,f(\xb))\}, \\
 \partial f(\xb) & = & \{\xi \in \R^n \mv (\xi,-1) \in N_{\epi f}(\xb,f(\xb))\}, \\
 \partial^{\infty} f(\xb) & = & \{\xi \in \R^n \mv (\xi,0) \in N_{\epi f}(\xb,f(\xb))\},
\end{eqnarray*}
where the last relation holds if $\epi f$ is locally closed around $(\xb,f(\xb))$
or, equivalently, $f$ is lower semicontinuous (lsc) around $\xb$, see e.g. \cite[Theorem 8.9]{RoWe98}.

Given a multifunction $M:\R^n\rightrightarrows \R^m$ and a point
$(\xb,\yb)\in \Gr M:=\{(x,y)\in \R^n \times \R^m\mv y\in M(x)\}$, where $\Gr M$ denotes the graph of $M$,
the {\em graphical derivative} of $M$ at $(\xb,\yb)$ is a multifunction
$D M(\xb,\yb):\R^n \rightrightarrows \R^m$ given by
\[D M(\xb,\yb)(u):=\{v \in \R^m \mv (u,v) \in T_{\Gr M}(\xb,\yb)\}.\]
The {\em regular coderivative} of $M$ at $(\xb,\yb)$ is a multifunction
$\hat D^\ast M(\xb,\yb):\R^m\rightrightarrows \R^n$ with the values
\[\hat D^\ast M(\xb,\yb)(\eta):=\{\xi \in \R^n\mv (\xi,-\eta) \in \hat N_{\Gr M}(\xb,\yb)\}\]
and the {\em limiting} ({\em Mordukhovich}) {\em coderivative} of $M$ at $(\xb,\yb)$ is a multifunction
$D^\ast M(\xb,\yb):\R^m\rightrightarrows \R^n$ defined by
\[D^\ast M(\xb,\yb)(\eta):=\{\xi\in \R^n\mv (\xi,-\eta)\in N_{\Gr M}(\xb,\yb)\},\]
i.e., $D^\ast M(\xb,\yb)(\eta)$ is the collection of all $\xi\in \R^n$ for which there are sequences
$(x_k,y_k)\to (\xb,\yb)$ and $(\xi_k,\eta_k)\to(\xi,\eta)$ with $(\xi_k, -\eta_k)\in \hat N_{\Gr M}(x_k,y_k)$,
or, equivalently, $\xi_k \in \hat D^\ast M(x_k,y_k)(\eta_k)$.
The notation $D M(\xb,\yb)$, $\hat D^\ast M(\xb,\yb)$ and $D^\ast M(\xb,\yb)$ is simplified to
$D M(\xb)$, $\hat D^\ast M(\xb)$ and $D^\ast M(\xb)$ when $M$ is single-valued at $\xb$ with $M(\xb) = \{\yb\}$.

Just as in case of subdiferentials and epigraphs, it is often important to have closed graphs of multifunctions.
We say that $M$ is {\em outer semicontinuous (osc) at} $\xb$ if the existence of
sequences $x_k \to \xb$ and $y_k \to y$ with $y_k \in S(x_k)$ implies $y \in S(\xb)$
and we say that $M$ is osc if it is osc at every point, which is equivalent to the closedness of $\Gr M$,
see \cite[Theorem 5.7]{RoWe98}.

For more details we refer to the monographs \cite{Mo06a,RoWe98}.

Directional versions of these limiting constructions were introduced in \cite{GM} and \cite{Gfr13a} for general Banach spaces
and later on equivalently reformulated for finite dimensional spaces in \cite{Gfr14a} in the following way.
Given a direction $u \in \R^n$, the {\em limiting} {\em normal cone to a subset} $\Omega\subset \R^n$
at $\xb\in\Omega$ {\em in direction} $u$ is defined by
\begin{equation*}\label{eq : NCdirdef}
  N_{\Omega}(\xb;u):=\{\xi \in\R^n\mv \exists (t_k)\downarrow 0,\ (u_k)\to u,\ (\xi_k)\to \xi:
  \xi_k\in \hat N_{\Omega}(\xb+t_ku_k) \ \forall k\}.
\end{equation*}

Note that by the definition we have $N_{\Omega}(\xb;0)=N_{\Omega}(\xb)$.
Further $N_{\Omega}(\xb;u) \subset N_{\Omega}(\xb)$ for all $u$, $N_{\Omega}(\xb;u)=\emptyset$ if $u \notin T_{\Omega}(\xb)$
and $N_{\Omega}(\xb;u)=\{0\}$ if $u \in T_{\Omega}(\xb) \setminus T_{\bd \Omega}(\xb)$.

The following simple lemma provides a hint about possible applications of directional limiting normal cones.

\begin{lemma} \label{lem : LimNorConeViaDir}
  Let $\Omega \subset \R^n$ be closed and $\xb \in\Omega$. Then
  \[N_{\Omega}(\xb) = \hat N_{\Omega}(\xb) \cup \bigcup_{u \in T_{\Omega}(\xb) \cap \mathbb{S}} N_{\Omega}(\xb,u).\]
\end{lemma}
\begin{proof}
  Inclusion $\supset$ follows directly from definition.

  Now let $\xi \in N_{\Omega}(\xb)$ and consider $x_k \to \xb$, $\xi_k \to \xi$ with
  $x_k \in \Omega$ and $\xi_k \in \hat N_\Omega(x_k)$. If $x_k = \xb$ for infinitely many $k$ we have
  $\xi \in \hat N_{\Omega}(\xb)$ due to closedness of $\hat N_{\Omega}(\xb)$. On the other hand if $x_k \neq \xb$
  for infinitely many $k$, we set $t_k := \norm{x_k - \xb}$ and $u_k := (x_k - \xb) / \norm{x_k - \xb}$
  and by passing to a subsequence we assume $(t_k) \downarrow 0$ and $u_k \to u \in \mathbb{S}$.
  Since $x_k = \xb + t_k u_k \in \Omega$ we conclude $\xi \in N_{\Omega}(\xb,u)$ as well as $u \in T_\Omega(\xb)$,
  completing the proof.
\end{proof}

\begin{proposition} \cite[Proposition 3.3]{YZ} \label{Pro : ProductSet}
  Let $\R^n$ be written as $\R^n = \R^{n_1} \times \ldots \times \R^{n_l}$ and for $x \in \R^n$ we write $x = (x_1,\ldots,x_l)$
  with $x_i \in \R^{n_i}$. Let $C_i \subset \R^{n_i}$ be closed for $i=1,\ldots,l$, set $C = C_1 \times \ldots \times C_l$
  and consider a point $\xb=(\xb_1,\ldots,\xb_l) \in C$ and a direction $h=(h_1,\ldots,h_l) \in \R^n$. Then
  \begin{equation*}
    N_C(\xb;h) \subset N_{C_1}(\xb_1;h_1) \times \ldots \times N_{C_l}(\xb_l;h_l).
  \end{equation*}
\end{proposition}

For a multifunction $M:\R^n\rightrightarrows \R^m$ and a direction $(u,v)\in \R^n\times \R^m$,
the {\em limiting} {\em coderivative} of $M$ {\em in direction} $(u,v)$ at $(\xb,\yb)\in \Gr M$ is defined as the multifunction
$D^\ast M((\xb,\yb);(u,v)):\R^m\rightrightarrows\R^n$ given by
\[D^\ast M((\xb,\yb);(u,v))(\eta):=\{\xi\in \R^n\mv (\xi,-\eta)\in N_{\Gr M}((\xb,\yb);(u,v))\}.\]

Clearly, one has $D^\ast M((\xb,\yb);(0,0))=D^\ast M(\xb,\yb)$.
In case of a continuously differentiable single-valued mapping, the following relations hold.
\begin{remark} \label{Rem: IdentityMapping}
  Let $F : \R^n \to \R^m$ be continuously differentiable and let $\nabla F(\xb)$ denote its Jocobian.
  One has $DF(\xb)(u) = \nabla F (\xb)u$ and thus $D^*F(\xb;(u,v))(\eta) \neq \emptyset$ if
  and only if $v = \nabla F (\xb) u$, in which case
  \[D^*F(\xb;(u,v))(\eta) = D^*F(\xb)(\eta) = \hat D^*F(\xb)(\eta) = (\nabla F (\xb))^T \eta.\]
\end{remark}

Our approach to directional limiting subdiferentials differs from the one established in \cite{Gfr13a, GM, LoWaYa17},
where it is either defined or equivalently described as a limit of regular subdiferentials.
In the finite dimensional setting these definitions read as follows.
Given $f: \R^n \to \bar \R$, $\xb \in \dom f$ and a direction $u \in \R^n$, consider the set
\[\partial_a f(\xb,u) := \{\xi \in \R^n \mv \exists (t_k) \downarrow 0,\ (u_k)\to u,\ (\xi_k)\to \xi:
f(\xb + t_k u_k) \to f(\xb), \xi_k \in \hat \partial f(\xb + t_k u_k) \ \forall k \},\]
which we will call the {\em analytic limiting subdiferential} of $f$ at $\xb$ {\em in direction} $u$,
following the notation from \cite[Definition 1.83]{Mo06a}.\footnote{
It was pointed out in \cite{LoWaYa17} that in \cite{Gfr13a} the condition $f(\xb + t_k u_k) \to f(\xb)$ was omitted
from the definition and the same holds true for the definition from \cite{GM}.
Such definition leads e.g. to violation of $\partial_a f(\xb,0) = \partial f(\xb)$,
which is clearly undesirable and the omission of $f(\xb + t_k u_k) \to f(\xb)$ was most likely unintentional.}

In this paper, inspired by directional coderivatives, we consider a direction $(u,\mu) \in \R^{n+1}$
and define the {\em limiting subdiferential} of $f$ at $\xb$ {\em in direction} $(u,\mu)$ via
\begin{equation} \label{eq : SubdifDef}
  \partial f(\xb;(u,\mu)) := \{\xi \in \R^n \mv (\xi,-1) \in N_{\epi f}((\xb,f(\xb));(u,\mu))\}.
\end{equation}
The advantages of this definition are twofold: First, it leads to a finer analysis and second, there is
a close relationship between subdiferentials and normal cones which allows us to easily carry over the results
obtained for normal cones to subdiferentials. More detailed discussion about the two versions
of directional subdiferentials is presented at the beginning of Section 4.

Finally, we present some well-known properties of multifunctions as well as their directional counterparts.
In order to do so, following \cite{Gfr13a}, we define a directional neighborhood of (a direction) $u \in \R^n$.

Given a direction $u\in \R^n$ and positive numbers $\rho,\delta>0$, consider the set $\mathcal{V}_{\rho,\delta}(u)$ given by
\begin{equation} \label{EqDefNbhd}
\mathcal{V}_{\rho,\delta}(u):=\{z \in \rho \B \mv \big\Vert \norm{u} z- \norm{z} u \big\Vert\leq\delta \norm{z} \, \norm{u}\}.
\end{equation}
We say that a set $\mathcal V$ is a {\em directional neighborhood} of $u$ if there exist $\rho,\delta>0$ such that
$\mathcal{V}_{\rho,\delta}(u) \subset \mathcal V$. Moreover, we say that a sequence $x_k \in \R^n$ converges to some $\xb$
from direction $u \in \R^n$ if for every directional neighborhood $\mathcal V$ of $u$ we have $x_k \in \xb + \mathcal V$ for
sufficiently large $k$, or, equivalently, if there exist $(t_k) \downarrow 0$ and $u_k \to u$ with $x_k = \xb + t_k u_k$.

\begin{definition}
  Let $M: \R^n \rightrightarrows \R^m$ and $(\xb,\yb) \in \Gr M$. We say that
  $M$ is {\em metrically subregular} at $(\xb,\yb)$ provided there exist $\kappa > 0$ and
  a neighborhood $U$ of $\xb$ such that
  \[\dist{M^{-1}(\yb)}{x} \leq \kappa \dist{M(x)}{\yb} \ \forall x \in U.\]
  Given $u \in \R^n$, we say $M$ is {\em metrically subregular in direction} $u$ at $(\xb,\yb)$
  if there exists a directional neighborhood $\mathcal{U}$ of $u$ such that the above estimate
  holds for all $x \in \xb + \mathcal{U}$.
\end{definition}

It is well-known that metric subregularity of $M$ at $(\xb,\yb)$ is equivalent to calmness of $M^{-1}$ at $(\yb,\xb)$.
We say that $S: \R^m \rightrightarrows \R^n$ is {\em calm} at $(\yb,\xb) \in \Gr S$
provided there exist $\kappa > 0$ and neighborhoods $U$ of $\xb$ and $V$ of $\yb$ such that
\begin{equation} \label{eq : CalmnessDef}
  S(y) \cap U \subset S(\yb) + \kappa \norm{y - \yb}\B \ \forall y \in V.
\end{equation}

It is also known that neighborhood $U$ can be reduced (if necessary) in such a way that neighborhood $V$
can be replaced by the whole space $\R^m$, cf. \cite[Exercise 3H.4]{DonRo09}.

For our purposes it is, however, suitable to employ estimate \eqref{eq : CalmnessDef}
without this simplification and to introduce the directional calmness by replacing $V$
by $\yb + \mathcal{V}$, where $\mathcal{V}$ is the appropriate directional neighborhood.

Calmness, similarly as some other Lipschitzian stability properties enables us to estimate
the images of $S$ around $(\yb,\xb)$ via $S(\yb)$ and the respective calmness modulus $\kappa$. However,
what we actually need for our directional calculus is the opposite, i.e., we need to be able to provide
an estimate of $\xb$ in terms of $S(y)$ for $y$ close to $\yb$. This our need is reflected
in the following inner version of calmness.

\begin{definition}
  A set-valued mapping $S: \R^m \rightrightarrows \R^n$ is called
  {\em inner calm at $(\yb,\xb) \in \Gr S$ with respect to (w.r.t.)} $\Omega \subset \R^m$
  if there exist $\kappa > 0$ and a neighborhood $V$ of $\yb$ such that
  \[\xb \in S(y) + \kappa \norm{y - \yb}\B \quad \forall \, y \in V \cap \Omega.\]
  If in the above definition $V = \yb + \mathcal{V}$, where $\mathcal{V}$ is a directional neighborhood of a direction $v \in \R^m$,
  we say that $S$ possesses the {\em inner calmness} property at $(\yb,\xb)$ w.r.t. $\Omega$ {\em in direction $v$}.
\end{definition}

Note that inner calmness of $S$ at $(\yb,\xb) \in \Gr S$ w.r.t. $\dom S$ in direction $v$
exactly corresponds to the directional inner semicompactness of $S$ at $(\yb,\xb) \in \Gr S$ in direction $v$
from \cite[Definition 4.4]{LoWaYa17}. In literature one can find also several other names for this property,
such as, e.g., Lipschitz lower semicontinuity \cite{KlKu15} or recession with linear rate \cite{Io16}.

Apart from the notions of directional metric subregularity, calmness and inner calmness we will make use also
of inner semicompactness and semicontinuity.

Recall that $S$ is {\em inner semicompact at $\yb$ w.r.t.} $\Omega \subset \R^m$ if for every sequence $y_k \setto{\Omega} \yb$
there exists a subsequence $K$ of $\N$ and a convergent sequence $(x_k)_{k\in K}$ with $x_k \in S(y_k)$ for $k \in K$.
Given $\xb \in S(\yb)$, we say that $S$ is {\em inner semicontinuous at $(\yb,\xb)$ w.r.t.} $\Omega \subset \R^m$
if for every sequence $y_k \setto{\Omega} \yb$ there exists a subsequence $K$ of $\N$
and a sequence $(x_k)_{k\in K}$ with $x_k \to \xb$ and $x_k \in S(y_k)$ for $k \in K$.
If $\Omega=\R^m$, we speak only about inner semicompactness at $\yb$ and inner semicontinuity at $(\yb,\xb)$.
For more details we refer to \cite{Mo06a}.

The directional versions of inner semicompactness and semicontinuity are obtained by restricting our attention to
$y_k$ converging to $\yb$ from some direction $v$.
We point out here that in \cite[Definition 4.4]{LoWaYa17} the authors defined
the directional versions of inner semicompactness and semicontinuity
in such a way that it allows them to find a suitable direction $h$,
i.e., they control the rate of convergence $x_k \to \xb$ by requiring
the difference quotients $(x_k - \xb)/t_k$ either to be bounded or to converge to some prescribed $h$.
We believe, however, that it is not very suitable to call such properties
semicompactness and semicontinuity, as those requirements are clearly much stronger
and they are not implied by their non-directional counterparts, as also the authors admit.

Clearly, inner calmness implies both inner semicontinuity and semicompactness.

Interestingly, in \cite[Theorem 8]{GO3} Gfrerer and Outrata also investigated the estimate from
definition of inner calmness and established some sufficient conditions to ensure both, calmness and inner calmness,
of a class of solution maps.

Note that in case of a single-valued mapping $\varphi: \R^m \to \R^n$, calmness and inner calmness coincide
and they read as
\begin{equation} \label{eq : CalmnessDefSingValued}
  \norm{ \varphi(y) - \varphi(\yb)} \leq \kappa \norm{y - \yb} \quad \forall \, y \in V.
\end{equation}
Further, $\varphi$ is called {\em Lipschitz continuous near} $\yb$ {\em in direction} $v$ if the inequality
\[\norm{\varphi(y) - \varphi(y^{\prime})} \leq \kappa \norm{y - y^{\prime}} \ \forall \, y, y^{\prime} \in \yb + \mathcal{V}\]
is fulfilled with $\mathcal{V}$ being a directional neighborhood of $v$.
Note that Lipschitz continuity of $\varphi$ near $\yb$ in direction $v$ actually implies
Lipschitz continuity of $\varphi$ near every point $y \in \yb + \mathcal{V}$, $y \neq \yb$.

In construction of the directional limiting calculus one is confronted with the following issue. Given
a mapping $S: \R^m \rightrightarrows \R^n$, a point $\yb \in \R^m$, a direction $v \in \R^m$
and a sequence $y_k \to \yb$ from $v$ (i.e. $y_k = \yb + t_k v_k$ for some $(t_k) \downarrow 0$, $v_k \to v$),
we would like to identify not only a point $\xb \in S(\yb)$ with $x_k \to \xb$ for some $x_k \in S(y_k)$,
but also a direction $h \in \R^n$ such that $x_k = \xb + t_k h_k$ for some $h_k \to h$.

The task of finding an appropriate direction $h$ is related to the following sets.
Given sequences $(a_k) \in \R^n$ and $(t_k) \downarrow 0$, we set
\begin{align} \label{eqn : OmegaDef}
  &  \Gamma(a_k,t_k) & := & \ \{\omega \in \R^n \mv \exists \textrm{ a subsequence } K \textrm{ of } \N :
    a_k / t_k \to \omega \textrm{ when } k \in K \}, \\ \label{eqn : OmegaInfDef}
  &  \Gamma^{\infty}(a_k,t_k) & := & \ \{\omega \in \mathbb{S} \mv t_k / \norm{a_k} \to 0,
    \exists \textrm{ a subsequence } K \textrm{ of } \N : a_k / \norm{a_k} \to \omega \textrm{ when } k \in K \}.
  \end{align}
Note that exactly one of these sets is not empty, since $\Gamma(a_k,t_k) = \emptyset$ is equivalent to $t_k / \norm{a_k} \to 0$.
In the situation considered above sequences $a_k$ appear in form
$a_k \in S(\yb + t_k v_k) - S(\yb)$. Thus, if $\Gamma(a_k,t_k) \neq \emptyset$, one can clearly take
a suitable direction $h \in \Gamma(a_k,t_k)$, while in the other case one can still proceed with
$h \in \Gamma^{\infty}(a_k,t_k)$ to obtain different (but rather rough) estimates. Notation
\eqref{eqn : OmegaDef},\eqref{eqn : OmegaInfDef} will be extensively used throughout the whole sequel.

Moreover, it is easy to see that inner calmness can be characterized in the following way.

\begin{lemma}
  A set-valued mapping $S: \R^m \rightrightarrows \R^n$ is inner calm at $(\yb,\xb) \in \Gr S$ w.r.t. $\Omega$ in direction $v$
  if and only if for every $(t_k) \downarrow 0$, $v_k \to v$ with $\yb + t_k v_k \in \Omega$ there exist
  a subsequence $K$ of $\N$ and a sequence $(x_k)_{k\in K}$
  with $x_k \in S(\yb + t_k v_k)$ for $k\in K$ such that $\Gamma(x_k - \xb,t_k) \neq \emptyset$.
\end{lemma}

We conclude this preparatory section with a mention concerning qualification conditions used
in the calculus being developed. Analogously to \cite{IoOut08}, our main qualification condition
will be the directional metric subregularity of the so-called feasibility mapping
associated with the considered calculus rule.
This mapping has typically the form $F(x)= \Omega - \varphi(x)$, where
$\Omega$ is a closed subset of $\R^m$ and $\varphi: \R^n \to \R^m$ is a continuous mapping.
A tool for verifying directional metric subregularity of such mappings for continuously differentiable $\varphi$
was recently established by Gfrerer and Klatte \cite[Corollary 1]{GfrKl16}
and we slightly extend this result here by allowing functions $\varphi$ to be just calm
in the prescribed direction.

\begin{proposition} \label{Pro : MetrSubregVerif}
  Let multifunction $F : \R^n \rightrightarrows \R^m$ be given by $F(x)= \Omega - \varphi(x)$,
  where $\varphi: \R^n \to \R^m$ is continuous and $\Omega \subset \R^m$ is a closed set.
  Further let $(\xb,0) \in \Gr F$ and $u \in \R^{n}$ be given and assume
  that $\varphi$ is calm at $\xb$ in direction $u$.
  Then $F$ is metrically subregular at $(\xb,0)$ in direction $u$ provided
  for all $w \in D\varphi(\xb)(u) \cap T_{\Omega}(\varphi(\xb))$ one has the implication
  \[0 \in D^*\varphi(\xb;(u,w))(\lambda), \,
  \lambda \in N_{\Omega}(\varphi(\xb);w) \ \Longrightarrow \ \lambda = 0.\]
\end{proposition}

The proof is based on the sum rule for coderivatives of multifunctions
and will be presented among other applications in Section 6.

\section{Calculus for directional limiting normal cones}

  Let $C \subset \R^n$ be a closed set and $\xb \in C$. If $\xb \in C \setminus \bd C$ then $N_C(\xb;h) = \{0\}$ for every $h \in \R^n$.
  Since $N_C(\xb;h) = \emptyset$ for $h \notin T_C(\xb)$
  and $N_C(\xb;h) = \{0\}$ for $h \in T_C(\xb) \setminus T_{\bd C}(\xb)$, it follows that for every set of directions
  $A \subset \R^n$ it holds that
  \begin{equation} \label{eq : ImportantDirections}
    \bigcup_{h \in A} N_C(\xb;h) = \bigcup_{h \in A \cap T_C(\xb)} N_C(\xb;h) = \bigcup_{h \in A \cap T_{\bd C}(\xb)} N_C(\xb;h).
  \end{equation}
  This observation allows us to consider only the indispensable directions in our estimates, as one can see in
  Theorems \ref{The : ConstSet} and \ref{The : ConstSetForw}.

\begin{theorem}[Pre-image sets] \label{The : ConstSet}
  Let $Q \subset \R^m$ be closed, consider a continuous function $\varphi: \R^n \to \R^m$ and set $C := \varphi^{-1}(Q)$.
  Assume further that the set-valued mapping $F: \R^n \rightrightarrows \R^m$ given by $F(x) = Q - \varphi(x)$ is
  metrically subregular at $(\xb,0)$ in some direction $h \in \R^n$. Then
  \[N_C(\xb;h) \subset \Big( \bigcup\limits_{v \in D\varphi(\xb)(h) \atop \cap T_Q(\varphi(\xb))}
  D^*\varphi(\xb;(h,v)) N_Q(\varphi(\xb);v) \Big) \cup
  \Big( \bigcup\limits_{v \in D\varphi(\xb)(0) \cap \mathbb{S} \atop \cap T_Q(\varphi(\xb))}
    D^*\varphi(\xb;(0,v)) N_Q(\varphi(\xb);v) \Big).\]
  Moreover, if $\varphi$ is calm at $\xb$ in direction $h$ we obtain a better estimate
  \[N_C(\xb;h) \subset \bigcup\limits_{v \in D\varphi(\xb)(h) \atop \cap T_Q(\varphi(\xb))}
  D^*\varphi(\xb;(h,v)) N_Q(\varphi(\xb);v).\]
\end{theorem}
\begin{proof}
  Let $x^* \in N_C(\xb;h)$ and consider $(t_k) \downarrow 0$, $h_k \to h$, $x_k^* \to x^*$ with
  $x_k:=\xb + t_k h_k \in C$ and $x_k^* \in \hat N_C(x_k)$.
  Since $x_k^* \in \hat N_C(x_k)$, for a fixed $k$ and for every
  $\varepsilon > 0$ there exists a real $r_\varepsilon > 0$ such that $\skalp{x_k^*,x-x_k} \leq \varepsilon \norm{x-x_k}$ for all
  $x \in C \cap \B_{r_\varepsilon}(x_k)$.
  Subregularity assumption yields existence of directional neighborhood $\mathcal{U}$ of $h$ and $\kappa > 0$ such that
  $\dist{C}{x} \leq \kappa \dist{Q}{\varphi(x)}$
  holds for all $x \in \xb + \mathcal{U}$ and for given sufficiently large $k$ and given $\varepsilon$
  we choose $r_\varepsilon$ such that $\B_{r_\varepsilon/2}(x_k) \subset \xb + \mathcal{U}$.

  Next we claim that for all $x \in \B_{r_{\varepsilon}/2}(x_k)$ it holds that
  \[\varepsilon \norm{x-x_k} - \skalp{x_k^*,x-x_k} + (\norm{x_k^*} + \varepsilon) \kappa \dist{Q}{\varphi(x)} \geq 0.\]
  Indeed, for $x \in \B_{r_{\varepsilon}/2}(x_k)$ we have $\norm{x-x_k} \leq r_{\varepsilon}/2$ and hence there exists
  $\tilde x \in C \cap \B_{r_\varepsilon}(x_k)$ with $\norm{x- \tilde x} = \dist{C}{x}$. Thus,
  \[\skalp{x_k^*,x-x_k} - \varepsilon \norm{x-x_k} \leq
  (\norm{x_k^*} + \varepsilon) \dist{C}{x} + \skalp{x_k^*,\tilde x-x_k} - \varepsilon \norm{\tilde x-x_k}
  \leq (\norm{x_k^*} + \varepsilon) \kappa \dist{Q}{\varphi(x)},\]
  showing the claimed inequality.

  Now we consider $\varepsilon_k \downarrow 0$ and conclude that
  $(x_k, \varphi(x_k), \varphi(x_k))$ is a local solution of the problem
  \begin{equation}\label{eq : AuxProb}
\min f(x,y,q) := \varepsilon_k \norm{x-x_k} - \skalp{x_k^*,x-x_k} + (\norm{x_k^*} + \varepsilon_k) \kappa \norm{y - q}
  \quad \textrm{s.t.} \quad (x,y,q) \in \Gr \varphi \times Q.
   \end{equation}
  The fuzzy optimality conditions for problem \eqref{eq : AuxProb}, cf. \cite[Theorem 2.7]{BZ}, \cite[Lemma 2.32]{Mo06a},
  state that  to every $\eta_k > 0$ there exist triples
  $(x_{i,k}, y_{i,k}, q_{i,k}) \in (x_k, \varphi(x_k), \varphi(x_k)) + \eta_k \B, \ i =1,2,$ with
  $\vert f(x_{1,k}, y_{1,k}, q_{1,k}) - f(x_k, \varphi(x_k), \varphi(x_k)) \vert \leq \eta_k$ and
  $(x_{2,k}, y_{2,k}, q_{2,k}) \in \Gr \varphi \times Q$ such that
  \begin{eqnarray}
 \left [ \begin{array}{l}
  x^{*}_{k} - \varepsilon_{k}\xi_{k}\\
  -(\norm{ x^{*}_{k}} + \varepsilon_{k})\kappa \nu_{k}
  \end{array} \right ]
   & \in &  \hat N_{\Gr \varphi} (x_{2,k}, y_{2,k})+\eta_{k}\B, \label{eqn : Fuzzy1}\\
 (\norm{ x^{*}_{k}} + \varepsilon_{k})\kappa \nu_{k} & \in & \hat N_{Q}(q_{2,k})+\eta_k \B   \label{eqn : Fuzzy2}
  \end{eqnarray}
  for some $\xi_k,\nu_k \in \B$. We take $\eta_k := t_k^2 \downarrow 0$ and consider the limiting process for $k \rightarrow \infty$.
  Since $(\norm{x_k^*} + \varepsilon_k) \kappa \nu_k$ is a bounded sequence, by passing to a subsequence we
  can assume that $(\norm{x_k^*} + \varepsilon_k) \kappa \nu_k $ converges to some $\bar z$.

  Denoting $a_k := \varphi(\xb + t_k h_k)-\varphi(\xb)$, we define direction $v$ to be an element of either
  $\Gamma(a_k,t_k)$ or $\Gamma^{\infty}(a_k,t_k)$, see \eqref{eqn : OmegaDef}-\eqref{eqn : OmegaInfDef}.
  Let us first consider the case $v \in \Gamma(a_k,t_k)$ and assume $a_k/t_k \to v$.
  We show that $\bar z \in N_Q(\varphi(\xb);v)$ and $x^* \in D^*\varphi(\xb;(h,v)) (\bar z)$.
  By virtue of \eqref{eqn : Fuzzy2} there is a sequence of vectors $z_k \in \hat N_{Q}(q_{2,k})$ such that
  \[(\norm{x_k^*} + \varepsilon_k) \kappa \nu_k \in z_k + \eta_k \B.\]
  Since $\eta_k \to 0$, we obtain $z_k \to \bar z$. Taking into account that, by virtue of the fuzzy optimality conditions,
   \[\norm{q_{2,k} - \varphi(x_k)} \leq \eta_k = t_k^2,\]
    we obtain
  \begin{equation} \label{eq : ConvOfFrac}
    \norm{(q_{2,k} - \varphi(\xb))/t_k - v} \leq \norm{(q_{2,k} - \varphi(x_k))/t_k}
    + \norm{(\varphi(\xb + t_k h_k) - \varphi(\xb))/t_k- v} \to 0.
  \end{equation}
  Since $\varphi(\xb) + t_k(q_{2,k} - \varphi(\xb))/t_k = q_{2,k} \in Q$, the claimed relation
  $\bar z \in N_Q(\varphi(\xb);v)$ follows.

  In order to show the second claim, we observe that $x_k^* - \varepsilon_k \xi_k \to x^*$
  and exploit in the same way as above relation \eqref{eqn : Fuzzy1}
  to show the existence of $(w_k,-u_k) \in \hat N_{\Gr \varphi}(x_{2,k}, y_{2,k})$
  with $(w_k,-u_k) \to (x^*,-\bar z)$. Again, one has that
  \[
  \norm{y_{2,k} - \varphi(x_{k})} \leq \eta_{k}= t^{2}_{k}
  \]
  and hence relation \eqref{eq : ConvOfFrac} holds with $q_{2,k}$ replaced by $y_{2,k}$.
  It follows that $(y_{2,k} - \varphi(\xb))/t_k \to v$ and similarly we conclude also $(x_{2,k} - \xb)/t_k \to h$.
  Thus, again, since $((\xb,\varphi(\xb)) + t_k(x_{2,k} - \xb)/t_k,(y_{2,k} - \varphi(\xb))/t_k) = (x_{2,k}, y_{2,k}) \in \Gr \varphi$,
  we obtain $(x^*,-\bar z) \in N_{\Gr \varphi}(\xb,\varphi(\xb);(h,v))$ and hence
  $x^* \in D^*\varphi(\xb;(h,v)) (\bar z)$ with $\bar z \in N_Q(\varphi(\xb);v)$.

  Finally, we consider the case $v \in \Gamma^{\infty}(a_k,t_k)$, assume $a_k/\norm{a_k} \to v$
  and show that $\bar z \in N_Q(\varphi(\xb);v)$ and $x^* \in D^*\varphi(\xb;(0,v)) (\bar z)$.
  Note that in this case we have $v \in \mathbb S$ and
  $t_k/ \norm{a_k} \to 0$ implying $t_k < \norm{a_k}$ for sufficiently large $k$.
  Hence, we proceed as in the previous case with $t_k$ replaced by $\norm{a_k}$
  and obtain the same result, the only difference being $(x_{2,k} - \xb)/\norm{a_k} \to 0$,
  showing the claimed relations.
  Observation \eqref{eq : ImportantDirections} now completes the proof of the first statement.

  The calmness assumption yields boundedness of $(\varphi(\xb + t_k h_k)-\varphi(\xb))/t_k$ and hence
  we always have $v \in \Gamma(a_k,t_k) \neq \emptyset$ and thus we only need to consider the first case.
  The proof is complete.
\end{proof}

\begin{corollary} \label{Cor : Cap}
  Let $C_i \subset \R^n, i=1,\ldots,l$,  be closed and set $C=\cap_{i=1}^l C_i$.
  Assume further that the set-valued mapping $F: \R^n \rightrightarrows \R^{nl}$ given by $F(x) = \prod_{i=1}^l (C_i - x)$ is
  metrically subregular at $(\xb,(0,\ldots,0))$ in some direction $h \in \R^n$. Then
  \[N_C(\xb;h) \subset \sum_{i=1}^l N_{C_i}(\xb;h).\]
\end{corollary}
\begin{proof}
  Note that $C=\varphi^{-1}(Q)$ for $\varphi : \R^n \to \R^{nl}$ given by $\varphi(x) := (x,\ldots,x)$($l$ copies) and
  $Q := \prod_{i=1}^l C_i$. Hence, the statement follows from Proposition \ref{Pro : ProductSet}, Theorem \ref{The : ConstSet}
  and Remark \ref{Rem: IdentityMapping}.
\end{proof}

Taking into account Proposition \ref{Pro : MetrSubregVerif}, as well as Remark \ref{Rem: IdentityMapping}
and Proposition \ref{Pro : ProductSet}, the assumption of metric subregularity of $F$ is implied by the condition
\begin{equation} \label{eq : capCQ}
\sum_{i=1}^l \lambda_i = 0, \lambda_i \in N_{C_i}(\xb;h) \ \Longrightarrow \ \lambda_i = 0.
\end{equation}

\begin{corollary}[Sets with constraint structure] \label{Cor : ConstSet+}
  Let $P \subset \R^n$ and $Q \subset \R^m$ be closed, consider a continuous function $\varphi: \R^n \to \R^m$
  and set $C := \{ x \in P \mv \varphi(x) \in Q\}$.
  Assume further that the set-valued mapping $F: \R^n \rightrightarrows \R^n \times \R^m$
  given by $F(x) = P \times Q - (x,\varphi(x))$ is metrically subregular at $(\xb,0,0)$ in some direction $h \in \R^n$. Then
  \begin{eqnarray*}
    N_C(\xb;h) & \subset & \Big( \bigcup\limits_{v \in D\varphi(\xb)(h) \atop \cap T_Q(\varphi(\xb))}
    D^*\varphi(\xb;(h,v)) N_Q(\varphi(\xb);v) + N_P(\xb; h) \Big) \\
    && \cup
    \Big( \bigcup\limits_{v \in D\varphi(\xb)(0) \cap \mathbb{S} \atop \cap T_Q(\varphi(\xb))}
    D^*\varphi(\xb;(0,v)) N_Q(\varphi(\xb);v) + N_P(\xb; 0) \Big).
  \end{eqnarray*}
  Moreover, if $\varphi$ is calm at $\xb$ in direction $h$ we obtain
  \[N_C(\xb;h) \subset \bigcup\limits_{v \in D\varphi(\xb)(h) \atop \cap T_Q(\varphi(\xb))}
    D^*\varphi(\xb;( h,v)) N_Q(\varphi(\xb);v) + N_P(\xb; h).\]
\end{corollary}
\begin{proof}
  It is sufficient to apply Theorem \ref{The : ConstSet} to the set
  $\tilde C = \Psi^{-1}(\tilde Q)$ at $\bar x$ with
  $\Psi(x) := (x,\varphi(x))$ and $\tilde Q = P \times Q$ and observe that
  $T_{\tilde Q}(\xb,\varphi(\xb)) \subset T_P(\xb) \times T_Q(\varphi(\xb))$ (\cite[Proposition 6.41]{RoWe98}).
\end{proof}

Assuming the calmness of $\varphi$ at $\xb$ in direction $h$, Proposition \ref{Pro : MetrSubregVerif} yields that the condition
\[\lambda_1 \in -D^*\varphi(\xb;(h,v))(\lambda_2) \cap N_{P}(\xb;h), \,
\lambda_2 \in N_{Q}(\varphi(\xb);v) \ \Longrightarrow \ \lambda_1, \lambda_2 = 0,\]
fulfilled for all $v \in D\varphi(\xb)(h) \cap T_{Q}(\varphi(\xb))$,
implies the required metric subregularity of $F$.

\begin{theorem}[Image sets] \label{The : ConstSetForw}
  Consider a closed set $C \subset \R^n$ and a continuous mapping $\varphi: \R^n \to \R^l$,
  set $Q = \varphi(C)$ and consider $\yb \in Q$ and a direction $v \in \R^l$.
  Let $\Psi : \R^l \rightrightarrows \R^n$ be given by $\Psi(y) := \varphi^{-1}(y) \cap C$.
  Assume that $\varphi: \R^n \to \R^l$ is Lipschitz continuous near every $\xb \in \Psi(\yb)$
  in all directions $h \in D\Psi(\yb,\xb)(v)$ and $h \in D\Psi(\yb,\xb)(0) \cap \mathbb{S}$.
  If $\Psi$ is inner semicompact at $\yb$ w.r.t. $Q$ in direction $v$, then
  \begin{eqnarray}\label{eqn : ConstSetForw}
    N_Q(\yb;v) & \subset & \bigcup\limits_{\xb \in \Psi(\yb)} \bigg(
    \Big( \bigcup\limits_{h \in \{h \mv v \in D\varphi(\xb)(h) \} \atop \cap T_C(\xb)}
    \{y^* \mv D^*\varphi(\xb;(h,v))(y^*) \cap N_C(\xb;h) \neq \emptyset\} \Big) \bigg. \\ \nonumber
    && \bigg. \qquad \cup
    \Big( \bigcup\limits_{h \in \{h \in \mathbb{S} \mv 0 \in D\varphi(\xb)(h) \} \atop \cap T_C(\xb)}
    \{y^* \mv D^*\varphi(\xb;(h,0))(y^*) \cap N_C(\xb;h) \neq \emptyset\} \Big) \bigg).
  \end{eqnarray}
  Moreover, if there exists $\xb \in \Psi(\yb)$ such that $\Psi$ is inner semicontinuous at $(\yb,\xb)$ w.r.t. $Q$ in direction $v$,
  then estimate \eqref{eqn : ConstSetForw} holds with this $\xb$, i.e., the union over $\xb \in \Psi(\yb)$ is superfluous.
  Finally, if there exists $\xb \in \Psi(\yb)$ such that $\Psi$ is inner calm at $(\yb,\xb)$ w.r.t. $Q$ in direction $v$, then
  \eqref{eqn : ConstSetForw} reduces to
  \[N_Q(\yb;v) \subset \bigcup\limits_{h \in \{h \mv v \in D\varphi(\xb)(h) \} \atop \cap T_C(\xb)}
    \{y^* \mv D^*\varphi(\xb;(h,v))(y^*) \cap N_C(\xb;h) \neq \emptyset\}.\]
\end{theorem}
\begin{proof}
  Let $y^* \in N_Q(\yb;v)$ and consider $(t_k) \downarrow 0$, $v_k \to v$, $y_k^* \to y^*$ with
  $x_k \in \Psi(\yb + t_k v_k)$ and $y_k^* \in \hat N_Q(\varphi(x_k))$.
  Under the inner semicompactness of $\Psi$ we can take $x_k \in \Psi(\yb + t_k v_k)$ converging to some $\xb \in \Psi(\yb)$,
  taking also into account continuity of $\varphi$ and closedness of $C$.
  On the other hand, if we assume the existence of $\xb \in \Psi(\yb)$ such that $\Psi$ is inner semicontinuous
  at $(\yb,\xb)$ w.r.t. $Q$ in direction $v$, we can directly take $x_k \in \Psi(\yb + t_k v_k)$ with $x_k \to \xb$.

  Now $y_k^* \in \hat N_Q(\varphi(x_k))$ yields $\skalp{y_k^*,y-\varphi(x_k)} \leq o(\norm{y-\varphi(x_k)})$ for all $y \in Q$.
  Thus, taking into account $\skalp{0,x-x_k} = 0$ for all $x$ and $\norm{y-\varphi(x_k)} \leq \norm{(y,x)-(\varphi(x_k),x_k)}$,
  we obtain
  \[\skalp{(y_k^*,0),(y,x)-(\varphi(x_k),x_k)} \leq o(\norm{(y,x)-(\varphi(x_k),x_k)})\]
  for all $y,x \in \Gr \Psi$, showing $(y_k^*,0) \in \hat N_{\Gr \Psi}(\varphi(x_k),x_k)$.

  Denoting $a_k := x_k-\xb$, we define direction $h$ to be an element of either
  $\Gamma(a_k,t_k)$ or $\Gamma^{\infty}(a_k,t_k)$, see \eqref{eqn : OmegaDef}-\eqref{eqn : OmegaInfDef},
  and hence, taking also into account observation \eqref{eq : ImportantDirections}, we obtain
  \begin{equation}\label{eq : AuxEstim}
    (y^*,0) \in \Big( \bigcup_{h \in D\Psi(\yb,\xb)(v)} N_{\Gr \Psi}((\yb,\xb);(v,h)) \Big) \cup
    \Big( \bigcup_{h \in D\Psi(\yb,\xb)(0) \cap \mathbb{S}} N_{\Gr \Psi}((\yb,\xb);(0,h)) \Big).
  \end{equation}
  Note also that if $\Psi$ is inner calm at $(\yb,\xb)$ w.r.t. $Q$ in direction $v$,
  it is also inner semicontinuous at $(\yb,\xb)$ w.r.t. $Q$ in direction $v$
  and there exists $x_k \in \Psi(\yb + t_k v_k)$ such that $(x_k - \xb)/t_k$ is bounded.
  Hence we always have $h \in \Gamma(a_k,t_k)$ and the second term in estimate \eqref{eq : AuxEstim} is superfluous.

  Finally, $\Gr \Psi = \Gr \varphi^{-1} \cap (\R^l \times C)$ and we can use Corollary \ref{Cor : Cap}.
  We consider only the case $(y^*,0) \in \bigcup_{h \in D\Psi(\yb,\xb)(v)} N_{\Gr \Psi}((\yb,\xb);(v,h))$,
  because the other case is analogous. Note that condition \eqref{eq : capCQ} is fulfilled, because
  if $-(\lambda_1,\lambda_2) \in N_{\R^l \times C}((\yb,\xb);(v,h))$ we have $\lambda_1 = 0$ and
  $(0,\lambda_2) \in N_{\Gr \varphi^{-1}}((\yb,\xb);(v,h))$ means $\lambda_2 \in D^*\varphi(\xb;(h,v))(0)$,
  which implies $\lambda_2 = 0$ due to assumed Lipschitz continuity of $\varphi$ near $\xb$ in direction $v$.
  Hence, taking into account Proposition \ref{Pro : ProductSet}, we obtain the existence of $h \in D\Psi(\yb,\xb)(v)$,
  $(w^*,z^*) \in N_{\Gr \varphi}((\xb,\yb);(h,v))$ and $x^* \in N_C(\xb;h)$ such that
  $(y^*,0) = (z^*,w^*) + (0,x^*)$, i.e., $x^* \in D^*\varphi(\xb;(h,v))(y^*) \cap N_C(\xb;h)$.
  This, together with \eqref{eq : ImportantDirections}, completes the proof.
\end{proof}

\begin{remark} \label{Rem : ConstSetForw}
  Note that the directional Lipschitz continuity of $\varphi$ is only needed to justify the usage of Corollary \ref{Cor : Cap}.
  Thus, it can be weakened by assuming that for every $\xb \in \Psi(\yb)$ and all directions
  $h \in D\Psi(\yb,\xb)(v)$ and $h \in D\Psi(\yb,\xb)(0) \cap \mathbb{S}$, the mapping
  \[(y,x) \rightrightarrows (\Gr \varphi^{-1} - (y,x)) \times (\R^l \times C - (y,x))\]
  is metrically subregular at $((\yb,\xb),(0,0,0,0))$ in directions $(v,h)$ and $(0,h)$.

  Moreover, the inner semicompactness of $\Psi$ at $\yb$ w.r.t. $Q$ in direction $v$ is clearly implied
  by the assumption that $\Psi((\yb + \mathcal{V}) \cap Q) = \varphi^{-1}((\yb + \mathcal{V}) \cap Q) \cap C$ is bounded
  for some directional neighborhood $\mathcal{V}$ of $v$. An analogous assumption was used in the standard
  version of this result in \cite[Theorem 6.43]{RoWe98}.
\end{remark}

\begin{proposition} \label{Prop : Cup}
  Let $C_i \subset \R^n, i=1,\ldots,l$, be closed and set $C=\cup_{i=1}^l C_i$.
  For $\xb \in C$ and $h \in \R^n$ denote $I(\xb) := \{i =1,\ldots,l \mv \xb \in C_i\}$
  and $I(\xb,h) := \{i \in I(\xb) \mv h \in T_{C_i}(\xb)\}$. Then
  \begin{equation} \label{eq : cup}
    N_C(\xb;h) \subset \bigcup_{i \in I(\xb,h)} N_{C_i}(\xb;h).
  \end{equation}
\end{proposition}
\begin{proof}
  Let $x^* \in N_C(\xb;h)$ and consider $(t_k) \downarrow 0$, $h_k \to h$, $x_k^* \to x^*$ with
  $\xb + t_k h_k \in C$ and $x_k^* \in \hat N_C(\xb + t_k h_k)$. Then
  there exists $i$ such that $\xb + t_k h_k \in C_i$ for infinitely many $k$, showing $i \in I(\xb,h)$.
  Since $C_i \subset C$, by passing to subsequence if necessary,
  we obtain $x_k^* \in \hat N_{C_i}(\xb + t_k h_k)$, showing $x^* \in N_{C_i}(\xb;h)$.
\end{proof}

\section{Calculus for directional limiting subdifferentials}

In this section we carry over the results for normal cones from the previous section to
directional limiting subdifferentials defined via normals to the epigraph by \eqref{eq : SubdifDef}.
However, we start by a brief discussion about the relations between the analytic directional limiting subdifferential
and the one given by \eqref{eq : SubdifDef}.

Consider the following simple example.
\begin{example}
  Let $\xb=0$, $h=1$ and $f: \R \to \R$ be given by
  \[f(x) = \left\{
	    \begin{array}{rl}
	      0 & \textrm{ if } x \leq 0, \\
	      1/\sqrt{k} & \textrm{ if } x \in (1/(2k),1/k] \textrm{ for } k \in \N, \\
	      1 & \textrm{ if } x > 1.
	    \end{array} \right.\]
  Then clearly $Df(\xb)(h) = \emptyset$ and thus $\partial f(\xb;(h,\nu)) = \emptyset$ for all $\nu$. On the other hand,
  $0 \in \hat \partial f(1/k)$ for all $k \in \N$ and thus $0 \in \partial_a f(\xb;h)$. \hfill$\triangle$
\end{example}

In order to better understand the difference between the two concepts of directional subdifferentials,
given an lsc function $f: \R^n \to \bar\R$, $\xb \in \dom f$ and a direction $h \in \R^n$,
we consider the following sets
\begin{eqnarray*}
  N_{\epi f}((\xb,f(\xb));(h,\pm \infty))
  &:=& \{(x^*,\beta) \mv \exists (t_k) \downarrow 0, (h_k,\nu_k) \to (h,\pm \infty), (x_k^*,\beta_k) \to (x^*,\beta): \\
  && \ t_k \nu_k \to 0, (x_k^*,\beta_k) \in \hat N_{\epi f}((\xb,f(\xb)) + t_k(h_k,\nu_k))\}.
\end{eqnarray*}

\begin{proposition} \label{Pro : SubdifVSNormCone}
  Let $f: \R^n \to \bar\R$ be finite at $\xb$ and lsc and consider $h \in \R^n$. One has
  \[\partial_a f(\xb;h) = \Big\{x^* \mv (x^*,-1) \in \bigcup\limits_{\nu \in Df(\xb)(h)} N_{\epi f}((\xb,f(\xb));(h,\nu))
  \cup N_{\epi f}((\xb,f(\xb));(h,\pm \infty))\Big\}.\]
\end{proposition}
\begin{proof}
  Let $x^* \in \partial_a f(\xb;h)$ and consider $(t_k) \downarrow 0$, $h_k \to h$, $x_k^* \to x^*$ with
  $f(\xb + t_k h_k) \to f(\xb)$ and $x_k^* \in \hat\partial f(\xb + t_k h_k)$, or, equivalently
  \[(x_k^*,-1) \in \hat N_{\epi f} (\xb + t_k h_k,f(\xb + t_k h_k)) =
  \hat N_{\epi f} ((\xb,f(\xb)) + t_k (h_k,\nu_k))\]
  for $\nu_k := (f(\xb + t_k h_k)-f(\xb))/t_k)$. Now if $\nu_k \to \pm \infty$,
  we immediately conclude $(x^*,-1) \in N_{\epi f}((\xb,f(\xb));(h,\pm \infty))$.
  On the other hand, if $\nu_k \nrightarrow \pm \infty$, there exists $\nu$ such that,
  after passing to a subsequence if necessary, we have $\nu_k \to \nu$. Thus,
  we conclude $(x^*,-1) \in N_{\epi f}((\xb,f(\xb));(h,\nu))$ and,
  taking into account $(\xb,f(\xb)) + t_k (h_k,\nu_k) \in \Gr f$, we obtain $\nu \in Df(\xb)(h)$.

  Now take $x^*$ fulfilling $(x^*,-1) \in N_{\epi f}((\xb,f(\xb));(h,\nu))$ for some $\nu \in Df(\xb)(h)$ or $\nu=\pm \infty$
  and consider $(t_k) \downarrow 0$, $(h_k,\nu_k) \to (h,\nu)$,
  $(x_k^*,\beta_k) \to (x^*,-1)$ with $(\xb,f(\xb)) + t_k(h_k,\nu_k) \in \epi f$ and
  $(x_k^*,\beta_k) \in \hat N_{\epi f}((\xb,f(\xb)) + t_k (h_k,\nu_k))$.
  Since $f(\xb) + t_k \nu_k > f(\xb + t_k h_k)$ implies $\beta_k = 0$, we assume
  that $f(\xb) + t_k \nu_k = f(\xb + t_k h_k)$ and thus obtain $-x_k^* / \beta_k \in \hat\partial f(\xb + t_k h_k)$
  and consequently $x^* \in \partial_a f(\xb;h)$.
  Moreover, $f(\xb + t_k h_k) \to f(\xb)$ follows directly from definition for $\nu=\pm \infty$
  and for $\nu \in Df(\xb)(h)$ it holds due to lsc of $f$ and $(\xb,f(\xb)) + t_k(h_k,\nu_k) \in \epi f$.
  This completes the proof.
\end{proof}

Note that one always has
\begin{equation}\label{eq : NoCalmnessInfEstim}
  N_{\epi f}((\xb,f(\xb));(h,\pm \infty)) \subset N_{\epi f}((\xb,f(\xb));(0,\pm 1)).
\end{equation}
and this is due to the fact that for $f(x_k) - f(\xb)$ and $t_k$ from definition of $N_{\epi f}((\xb,f(\xb));(h,\pm \infty))$
we have $\Gamma^{\infty}(f(x_k) - f(\xb),t_k) = \pm 1$.
Moreover, if $f$ is calm at $\xb$ in direction $h$, the sequence $(f(x_k) - f(\xb))/t_k$ is bounded
and thus $N_{\epi f}((\xb,f(\xb));(h,\pm \infty)) = \emptyset$.

Note also that the calmness of an extended real-valued function is always understood with respect to its domain
and hence does not exclude e.g. the indicator function of a set. Moreover, for our purposes,
we in fact only need the existence of $\varepsilon, \kappa > 0$ and a directional neighborhood $\mathcal{U}$ of $h$ such that
\begin{equation} \label{eq : AttentCalmness}
  \vert f(x) - f(\xb) \vert \leq \kappa \norm{x-\xb} \ \forall x \in \xb + \mathcal{U} \textrm{ with }
  \vert f(x) - f(\xb) \vert \leq \varepsilon,
\end{equation}
suggesting that discontinuities of $f$ also do not cause any harm. However, in order to keep the presentation as simple as possible,
in the sequel we will only consider the calmness.

\begin{corollary}\label{Cor : SubdifUnderCalmness}
  Let $f: \R^n \to \bar\R$ be finite at $\xb$ and consider $h \in \R^n$.
  Assume further that $f$ is calm at $\xb$ in direction $h$. Then one has
  \[\partial_a f(\xb;h) = \bigcup\limits_{\nu \in Df(\xb)(h)} \partial f(\xb;(h,\nu))
=\Big\{x^* \mv (x^*,-1) \in \bigcup\limits_{\nu \in Df(\xb)(h)} N_{\epi f}((\xb,f(\xb));(h,\nu))\Big\}.\]
  If $f$ is also directionally differentiable at $\xb$ in $h$,
  then $Df(\xb)(h)=f^{\prime}(\xb;h)$ and
  $\partial_a f(\xb;h) = \partial f(\xb;(h,f^{\prime}(\xb;h)))$.
\end{corollary}

This corollary shows that the results for directional limiting subdifferentials obtained later in this section can be easily carried over
to analytic directional limiting subdifferentials whenever the considered function $f$ is directionally calm.
In case $f$ fails to be calm, one can get the results for analytic directional subdifferentials using the estimate
\[\partial_a f(\xb;h) \subset \bigcup\limits_{\nu \in Df(\xb)(h)} \partial f(\xb;(h,\nu))
\cup \partial f(\xb;(0,-1)) \cup \partial f(\xb;(0,1)),\]
which follows from \eqref{eq : NoCalmnessInfEstim}.

Another possible approach to calculus for directional limiting normal cones and subdifferentials would
be to start with subdifferentials, build first the calculus for subdifferentials from the scratch and then carry it over
to normal cones. The role of the bridge between the two concepts could be played by
equivalent characterization of directional normal cones via directional subdifferentials of the indicator function
or the distance function. For the sake of completeness, we present these results now.

Given a closed set $C$, we consider a point $\xb \in C$ and a direction $h \in T_C(\xb)$.
Clearly, $\delta_C(\cdot)$ and $\dist{C}{\cdot}$ are calm and directionally differentiable at $\xb$ in $h$
with $\delta_C^{\prime}(\xb;h) = {\rm d}_C^{\prime}(\xb;h) = 0$.
Thus, taking into account Corollary \ref{Cor : SubdifUnderCalmness},
we can restrict our attention to the analytic subdifferentials.

While the relation $N_C(\xb,h) = \partial_a \delta_C(\xb;h)$ follows directly from definitions,
in order to deal with the distance function we need to consider the following lemma.

\begin{lemma}
  For $C$, $\xb$ and $h$ as above it holds that
  \begin{equation} \label{eq : EquivDefSubdDist}
    \partial_a \dist{C}{\xb;h} = \{x^* \mv \exists (t_k) \downarrow 0, h_k \to h, x_k^* \to x^*: \xb + t_k h_k \in C,
    x_k^* \in \hat\partial \dist{C}{\xb + t_k h_k} \}.
  \end{equation}
\end{lemma}
\begin{proof}
  Inclusion $\supset$ follows directly from definition. Now take $x^* \in \partial_a \dist{C}{\xb;h}$ and consider
  $(t_k) \downarrow 0$, $\tilde h_k \to h$, $x_k^* \to x^*$ with $x_k^* \in \hat\partial \dist{C}{\xb + t_k \tilde h_k}$.
  From \cite[Example 8.53]{RoWe98} we obtain
  $\hat\partial \dist{C}{\xb + t_k \tilde h_k} \subset \hat N_C(x_k) \cap \B = \hat\partial \dist{C}{x_k}$
  for every $x_k \in \proj{C}{\xb + t_k \tilde h_k}$. Taking some $x_k \in \proj{C}{\xb + t_k \tilde h_k}$
  and setting $h_k := (x_k - \xb)/t_k$ we obtain $\xb + t_k h_k = x_k \in C$. Moreover, since $h \in T_C(\xb)$ we conclude
  \[t_k \norm{h_k - \tilde h_k} = \dist{C}{\xb + t_k \tilde h_k} \leq \dist{C}{\xb + t_k h} + t_k \norm{\tilde h_k - h} = o(t_k),\]
  showing $h_k \to h$ and finishing the proof.
\end{proof}

\begin{corollary}
  Let $C \subset \R^n$ be a closed set, $\xb \in C$ and $h \in T_C(\xb)$. Then
  \[N_C(\xb;h) = \bigcup_{\alpha > 0} \alpha \partial_a \dist{C}{\xb;h}.\]
\end{corollary}
\begin{proof}
  Taking into account \eqref{eq : EquivDefSubdDist}, the claim follows from \cite[Corollary 1.96]{Mo06a}.
\end{proof}

In \cite{IoOut08}, Ioffe and Outrata used subdifferentials of the distance function as the starting point
for deriving the qualification conditions required for calculus rules. The previous lemma allows us to
state a directional counterpart to their basic tool, \cite[Proposition 3.1]{IoOut08}.

\begin{corollary}
  Given $C$, $\xb$ and $h$ as in the previous corollary,
  if $f: \R^n \to \bar \R$ is an lsc function fulfilling
  $f(x) = 0, \ \forall x \in C$ and $f(x) \geq \dist{C}{x}, \forall x \in \R^n$,
  then $\partial_a \dist{C}{\xb;h} \subset \partial_a f(\xb;h)$.
\end{corollary}
\begin{proof}
  The assumptions on $f$ imply $\hat\partial \dist{C}{x} \subset \hat\partial f(x)$ for every $x \in C$. Hence
  in \eqref{eq : EquivDefSubdDist} we obtain that $x_k^* \in \hat\partial f(\xb + t_k h_k)$ and the claim follows.
\end{proof}

Finally, given $f$ and $\xb$ as before and a direction $(h,\nu) \in \R^{n+1}$,
we introduce the {\em singular subdifferential} of $f$ at $\xb$ {\em in direction} $(h,\nu)$ as
\[\partial^{\infty} f(\xb;(h,\nu)) := \{x^* \in \R^n \mv (x^*,0) \in N_{\epi f}((\xb,f(\xb));(h,\nu))\}.\]
This notion will be used in qualification conditions which mimic their counterparts from the ``standard''
generalized differential calculus.
As it will be shown in Corollary \ref{Cor : SingDifLipsch} below,
$\partial^{\infty} f(\xb;(h,\nu)) \subset \{0\}$ if $f$ is Lipschitz continuous near $\xb$ in direction $h$.

\subsection{Chain rule and its corollaries}

We start this subsection by an auxiliary result concerning separable functions,
which plays a role in deriving the sum rule from the chain rule.
Note that, unlike the classical case \cite[Proposition 10.5]{RoWe98},
we need to impose some mild assumptions in order to obtain a reasonable estimate.

\begin{proposition}[Separable functions] \label{Pro : ProductSubgrad}
  Let $\R^n$ be decomposed as $\R^n = \R^{n_1} \times \ldots \times \R^{n_l}$ and let $x = (x_1,\ldots,x_l)$
  with $x_i \in \R^{n_i}$. Let $f_i : \R^{n_i} \to \bar \R$ be lsc for $i=1,\ldots,l$
  and let all but one of $f_i$ be calm at $\xb_i$ in direction $h_i$.
  Set $f(x) = f_1(x_1) + \ldots + f_l(x_l)$ and consider $\xb=(\xb_1,\ldots,\xb_l) \in \dom f$
  and some direction $(h,\nu)=(h_1,\ldots,h_l,\nu) \in \R^{n+1}$. Then
  \begin{eqnarray}
    \partial f(\xb;(h,\nu)) & \subset & \bigcup_{\nu_i \in Df_i(\xb_i)(h_i) \atop \nu_1+\ldots+\nu_l=\nu}
    (\partial f_1(\xb_1;(h_1,\nu_1)) \times \ldots \times \partial f_l(\xb_l;(h_l,\nu_l))), \\ \label{eqn : SingSubd}
    \partial^{\infty} f(\xb;(h,\nu)) & \subset & \bigcup_{\nu_i \in Df_i(\xb_i)(h_i) \atop \nu_1+\ldots+\nu_l=\nu}
    (\partial^{\infty} f_1(\xb_1;(h_1,\nu_1)) \times \ldots \times \partial^{\infty} f_l(\xb_l;(h_l,\nu_l))).
  \end{eqnarray}
\end{proposition}

  The proof follows from Theorem \ref{The : ConstSetForw}, since with $\varphi: \R^{n+l} \to \R^{n+1}$ given by
  $\varphi(x_1,\alpha_1, \ldots, x_l,\alpha_l) := (x_1, \ldots, x_l,\sum_{i=1}^l \alpha_i)$ we obtain
  $\epi f = \varphi(\prod_{i=1}^l \epi f_i)$.
  Moreover, it can be shown that
  \begin{equation} \label{eq : MappingDef}
  \Psi(x_1, \ldots, x_l,\alpha) :=
  \left\{(x_1,\alpha_1, \ldots, x_l,\alpha_l) \in {\textstyle\prod_{i=1}^l} \epi f_i \,\big\vert\,
  {\textstyle\sum_{i=1}^l} \alpha_i = \alpha\right\}
  \end{equation}
  is inner calm at $((\xb,f(\xb)),(\xb_1,f_1(\xb_1),\ldots,\xb_l,f_l(\xb_l)))$ w.r.t. $\epi f$ in direction $(h,\nu)$,
  due to the calmness of all but one of $f_i$, even when the calmness is considered in the sense of \eqref{eq : AttentCalmness}.
  For the sake of brevity, the technical details are skipped.

\begin{remark} \label{Rem : InnCalm}
  Clearly, the calmness of all but one $f_i$ is just a sufficient condition that can be replaced
  by requiring the inner calmness of mapping \eqref{eq : MappingDef}. Moreover,
  one can also apply Theorem \ref{The : ConstSetForw} without these assumptions
  to obtain more complicated estimates.
\end{remark}

We have also stated the result concerning singular subdifferentials \eqref{eqn : SingSubd}, because we will need it
for deriving qualification conditions for the sum rule.
Later on we will not write down the results for singular subdifferentials although
usually the proofs will be applicable to this case as well.

\begin{theorem}[Directional subdifferentials chain rule] \label{The : ChainRule}
  Let $\varphi: \R^n \to \R^m$ be continuous, $g: \R^m \to \bar\R$ be finite at $\varphi(\xb)$ and lsc and set $f=g \circ \varphi$.
  Given a direction $(h,\nu) \in \R^{n+1}$, assume further that the set-valued mapping
  $F: \R^{n+1} \rightrightarrows \R^{m+1}$ given by $F(x,\alpha) = \epi g - (\varphi(x),\alpha)$
  is metrically subregular at $((\xb,f(\xb)),(0,0))$ in direction $(h,\nu)$. Then
  \begin{eqnarray*}
    \partial f(\xb;(h,\nu)) & \subset &
    \Big( \bigcup\limits_{v \in \{w \in D\varphi(\xb)(h) \mv \atop \qquad \nu \in Dg(\varphi(\xb))(w)\}}
    D^*\varphi(\xb;(h,v)) \partial g(\varphi(\xb);(v,\nu)) \Big) \\
    && \cup \Big( \bigcup\limits_{v \in \{w \in D\varphi(\xb)(0) \cap \mathbb{S} \mv \atop \quad 0 \in Dg(\varphi(\xb))(w)\}}
    D^*\varphi(\xb;(0,v)) \partial g(\varphi(\xb);(v,0)) \Big).
  \end{eqnarray*}
  Moreover, if $\varphi$ is calm at $\xb$ in direction $h$, then
  \[\partial f(\xb;(h,\nu)) \subset \bigcup\limits_{v \in \{w \in D\varphi(\xb)(h) \mv \atop \qquad \nu \in Dg(\varphi(\xb))(w)\}}
  D^*\varphi(\xb;(h,v)) \partial g(\varphi(\xb);(v,\nu)).\]
\end{theorem}
\begin{proof}
  Take $(x^*,-1) \in N_{\epi f}((\xb,f(\xb));(h,\nu))$ and note that
  $\epi f = \Psi^{-1}(\epi g)$ for $\Psi(x,\alpha) := (\varphi(x),\alpha)$.
  Theorem \ref{The : ConstSet} yields the existence of either $(v,\mu) \in D\Psi(\xb,f(\xb))(h,\nu)$, such that
  \begin{equation*} \label{eq : RegOpt}
    (x^*,-1) \in D^*\Psi((\xb,f(\xb));((h,\nu),(v,\mu))) N_{\epi g}(\Psi(\xb,f(\xb));(v,\mu))
  \end{equation*}
  or $(\tilde v,\tilde \mu) \in D\Psi(\xb,f(\xb))(0,0) \cap \mathbb{S}$ with
  \begin{equation*} \label{eq : IrregOpt}
  (x^*,-1) \in D^*\Psi((\xb,f(\xb));((0,0),(\tilde v,\tilde \mu))) N_{\epi g}(\Psi(\xb,f(\xb));(\tilde v,\tilde \mu)).
  \end{equation*}
  In the first case, there exists $(y^*,\beta) \in N_{\epi g}(\Psi(\xb,f(\xb));(v,\mu))$ with
  \[(x^*,-1,-y^*,-\beta) \in N_{\Gr \Psi}((\xb,f(\xb),\varphi(\xb),f(\xb));(h,\nu,v,\mu)).\]
  Since $\Gr \Psi = H(\Gr \varphi \times \Gr \I)$ for bijective function $H(x,y,a,b)=(x,a,y,b)$,
  applying Theorem \ref{The : ConstSetForw} and Proposition \ref{Pro : ProductSet} we obtain
  $(x^*,-y^*) \in N_{\Gr \varphi}((\xb,\varphi(\xb));(h,v))$ and $(-1,-\beta) \in N_{\Gr \I}((f(\xb),f(\xb));(\nu,\mu))$.
  Note that the former relation also implies $v \in D\varphi(\xb)(h)$ and the latter yields
  $\beta = -1$ and $\mu=\nu$. Consequently $y^* \in \partial g(\varphi(\xb);(v,\nu))$ and $\nu \in Dg(\varphi(\xb))(v)$.

  On the other hand, in the case $(\tilde v,\tilde \mu) \in D\Psi(\xb,f(\xb))(0,0) \cap \mathbb{S}$
  we can proceed analogously with replacing $(h,\nu)$ by $(0,0)$ to obtain
  $x^* \in D^*\varphi(\xb;(0,\tilde v)) \partial g(\varphi(\xb);(\tilde v,0))$,
  $\tilde v \in D\varphi(\xb)(0)$, $0 \in Dg(\varphi(\xb))(\tilde v)$ and $\tilde \mu= 0$, showing also $\tilde v \in \mathbb{S}$.

  Since the calmness of $\varphi$ at $\xb$ in $h$ is equivalent to the calmness of $\Psi$ at $(\xb,f(\xb))$ in direction
  $(h,\nu)$, we obtain only the first possibility and hence the appropriate simpler estimate.
  This finishes the proof.
\end{proof}

Taking into account Proposition \ref{Pro : MetrSubregVerif} and the arguments from the proof
and assuming the calmness of $\varphi$ at $\xb$ in direction $h$, the condition
\[0 \in D^*\varphi(\xb;(h,v)) (\lambda), \, v \in D\varphi(\xb)(h), \, \lambda \in \partial^{\infty} g(\varphi(\xb);(v,\nu))
  \ \Longrightarrow \ \lambda = 0\]
implies the required directional metric subregularity of $F$.

Next consider $f: \R^n \times \R^l \to \bar \R$, a point $(\xb,\yb) \in \dom f$ and denote $f_x := f(\cdot,\yb)$. Then,
the {\em partial subdifferential} of $f$ with respect to $x$ at $\xb$ for $\yb$ in direction $(h,\nu)$
is given as $\partial_x f((\xb,\yb);(h,\nu)) := \partial f_x(\xb;(h,\nu))$.

\begin{corollary}
  Let $f: \R^n \times \R^l \to \bar \R$ be finite at $(\xb,\yb)$ and lsc.
  Given a direction $(h,\nu) \in \R^{n+1}$, assume further that
  the set-valued mapping $F: \R^{n+1} \rightrightarrows \R^{n+l+1}$ given by
  $F(x,\alpha) = \epi f - (x,\yb,\alpha)$ is metrically subregular at $((\xb,f(\xb,\yb)),(0,0,0))$ in $(h,\nu)$. Then
  \[\partial_x f((\xb,\yb);(h,\nu)) \subset \{x^* \mv \exists y^* \textrm{ with } (x^*,y^*) \in \partial f((\xb,\yb);(h,0,\nu))\}.\]
\end{corollary}
\begin{proof}
  Since $f(\cdot,\yb) = f \circ \varphi$ with differentiable $\varphi(x) = (x,\yb)$,
  Theorem \ref{The : ChainRule} yields the statement.
\end{proof}

From Proposition \ref{Pro : MetrSubregVerif} we infer that
the above imposed metric subregularity of $F$ is implied by the condition
\[(0,\lambda) \in \partial^{\infty} f((\xb,\yb);(h,0,\nu)) \ \Longrightarrow \ \lambda = 0.\]

\begin{corollary}[Directional subdifferential sum rule]\label{Cor : SumRule}
  Suppose $f=f_1 + \ldots + f_l$, where $f_i : \R^n \to \bar \R$ are finite at $\xb$, lsc
  and all but one are calm at $\xb$ in direction $h \in \R^n$.
  Assume further that the set-valued mapping
  $F: \R^{n+1} \rightrightarrows \R^{nl+1}$ given by $F(x,\alpha) = \epi g - (x,\ldots,x,\alpha)$ with $g : \R^{nl} \to \bar \R$
  defined via $g(x_1,\ldots,x_l) = f_1(x_1) + \ldots + f_l(x_l)$, is metrically subregular at $((\xb,f(\xb)),(0,\ldots,0,0))$
  in direction $(h,\nu)$. Then
  \[\partial f(\xb;(h,\nu)) \subset \partial f_1(\xb;(h,\nu)) + \ldots + \partial f_l(\xb;(h,\nu)).\]
\end{corollary}
\begin{proof}
  Note that $f=g \circ \varphi$ for $\varphi : \R^n \to \R^{nl}$ given by $\varphi(x) := (x,\ldots,x)$.
  Now the proof follows from
  Theorem \ref{The : ChainRule}, Proposition \ref{Pro : ProductSubgrad} and Remark \ref{Rem: IdentityMapping}.
\end{proof}

Again, taking into account Propositions \ref{Pro : MetrSubregVerif} and \ref{Pro : ProductSubgrad},
the metric subregularity of $F$ from Corollary \ref{Cor : SumRule} is implied by the condition
\begin{equation} \label{eq : SumCQ}
  \sum_{i=1}^l \lambda_i = 0, \lambda_i \in \partial^{\infty} f_i(\xb,(h,\nu)) \ \Longrightarrow \ \lambda_i = 0.
\end{equation}
Inclusion \eqref{eq : SumCQ} holds true in particular if all but one of $f_i$ are Lipschitz continuous near $\xb$ in direction $h$.

Naturally, Remark \ref{Rem : InnCalm} applies here as well.

\subsection{Directional limiting subdifferentials of special functions}

We conclude this section with estimates for directional limiting subdifferentials of
the pointwise maximum and minimum of a finite family of functions, the distance function and the value function.

\begin{proposition} \label{Pro : MaxSubdif}
  Suppose $f = \max \{f_1, \ldots, f_l\}$ for functions $f_i : \R^n \to \bar \R$ that are continuous at $\xb$.
  Given a direction $(h,\nu) \in \R^{n+1}$, assume further that the set-valued mapping
  $F: \R^{n+1} \rightrightarrows \R^{l(n+1)}$ given by $F(x,\alpha) = \prod_{i=1}^l (\epi f_i - (x,\alpha))$
  is metrically subregular at $((\xb,f(\xb)),((0,0),\ldots,(0,0))$ in $(h,\nu)$.
  Then one has
  \[\partial f(\xb;(h,\nu)) \subset \bigcup_{J \subset I_0(\bar x,(h,\nu))}
  \Big( \co \{\partial f_i(\xb;(h,\nu)) \mv i \in J \}
  + \textstyle\sum_{i \in I(\bar x,(h,\nu)) \setminus J} \partial^{\infty} f_i(\xb;(h,\nu)) \Big),\]
  where
  \begin{eqnarray}
    I(\bar x,(h,\nu)) & := & \{i \mv f(\xb) = f_i(\xb) \textrm{ and } \nu \in Df_i(\xb)h\}, \nonumber \\ \label{eqn : I_0def}
    I_0(\bar x,(h,\nu)) & := & \{i \in I(\bar x,(h,\nu)) \mv \partial f_i(\xb;(h,\nu)) \neq \emptyset\}.
  \end{eqnarray}
\end{proposition}
\begin{proof}
  Since $\epi f= \bigcap_{i=1}^l \epi f_i$, Corollary \ref{Cor : Cap}
  yields that for $(x^*,-1) \in N_{\epi f}((\xb,f(\xb)),(h,\nu))$ there exists
  $(x_i^*,-\beta_i) \in N_{\epi f_i}((\xb,f(\xb)),(h,\nu))$ such that
  $(x^*,-1) = \sum_{i=1}^l (x_i^*,-\beta_i)$. If $f_i(\xb) < f(\xb)$ for some $i$, we have $(x_i^*,-\beta_i) = (0,0)$
  due to continuity of $f_i$ and the same holds true if $\nu \notin Df_i(\xb)h$. Hence we consider only $i \in I(\bar x,(h,\nu))$.

  Given a function $g$ and a point $x \in \dom g$, one always has that $(y^*,-\beta) \in \hat N_{\epi g}(x,g(x))$
  implies $\beta \geq 0$ and hence we obtain $\beta_i \geq 0$ for all $i \in I(\bar x,(h,\nu))$.
  Setting $J := \{i \in I(\bar x,(h,\nu)) \mv \beta_i > 0 \}$, for $i \in J$
  we obtain $\beta_i (x_i^* / \beta_i,-1) \in N_{\epi f_i}((\xb,f_i(\xb)),(h,\nu))$
  and thus $x_i^* \in \beta_i \partial f_i(\xb;(h,\nu))$, showing also $J \subset I_0(\bar x,(h,\nu))$.
  On the other hand, for $i \notin J$ we have $\beta_i=0$ and hence $x_i^* \in \partial^{\infty} f_i(\xb;(h,\nu))$.
  This completes the proof.
\end{proof}

Taking into account \eqref{eq : capCQ}
and the mentioned fact that $(y^*,-\beta) \in \hat N_{\epi g}(x,g(x))$ implies $\beta \geq 0$,
the required metric subregularity of $F$ from Proposition \ref{Pro : MaxSubdif} is again implied by condition \eqref{eq : SumCQ}.

\begin{proposition} \label{Pro : MinSubdif}
  Suppose $f = \min \{f_1, \ldots, f_l\}$ for lsc functions $f_i : \R^n \to \bar \R$ and $\xb \in \dom f$.
  Given a direction $(h,\nu) \in \R^{n+1}$, consider index set $I_0(\bar x,(h,\nu))$ given by \eqref{eqn : I_0def}. Then one has
  \[\partial f(\xb;(h,\nu)) \subset \bigcup_{i \in I_0(\bar x,(h,\nu))} \partial f_i(\xb;(h,\nu)).\]
\end{proposition}
\begin{proof}
  Since $\epi f= \bigcup_{i=1}^l \epi f_i$, Proposition \ref{Prop : Cup} yields that for
  $(x^*,-1) \in N_{\epi f}((\xb,f(\xb)),(h,\nu))$ there exists $i$ such that
  $(x^*,-1) \in N_{\epi f_i}((\xb,f(\xb)),(h,\nu))$. Moreover, $(\xb,f(\xb)) \in \epi f_i$
  implies $f_i(\xb) \leq f(\xb)$ while $f_i(\xb) \geq f(\xb)$ follows from definition of $f$
  and consequently $x^* \in \partial f_i(\xb;h)$. This verifies that $i \in I_0(\bar x,h)$ and completes the proof.
\end{proof}

\begin{theorem}[Directional subdifferentials of value function] \label{The : PartialMin}
  Consider an lsc function $f: \R^n\times\R^l \to \bar\R$, set $\vartheta(y) = \inf_{x \in \R^n} f(x,y)$
  and assume that $\vartheta$ is finite at $\yb$.
  Let $S: \R^l \rightrightarrows \R^n$ be the solution mapping given by $S(y) = \argmin f(\cdot,y)$
  and consider a direction $(v,\mu) \in \R^{l+1}$.
  If $S$ is inner semicompact at $\yb$ in direction $v$, then
  \begin{eqnarray*}
    y^* \in \partial \vartheta(\yb;(v,\mu)) & \Longrightarrow &
    (0,y^*) \in \bigcup\limits_{\xb \in S(\yb)} \bigg(
    \Big( \bigcup\limits_{h \in \{h \mv \mu \in Df(\xb,\yb)(h,v)\}} \partial f((\xb,\yb);(h,v,\mu)) \Big) \bigg. \\
    && \bigg. \hspace{2.2cm} \cup \ \Big( \bigcup\limits_{h \in \{h \in \mathbb{S} \mv 0 \in Df(\xb,\yb)(h,0)\}}
    \partial f((\xb,\yb);(h,0,0)) \Big) \bigg).
  \end{eqnarray*}
  Moreover, if there exists $\xb \in S(\yb)$ such that $S$ is inner semicontinuous at $(\yb,\xb)$ in direction $v$,
  then the previous estimate holds with this $\xb$, i.e., the union over $\xb \in S(\yb)$ is superfluous.
  Finally, if there exists $\xb \in S(\yb)$ such that $S$ is inner calm at $(\yb,\xb)$ in direction $v$, then
  the estimate reduces to
  \[y^* \in \partial \vartheta(\yb;(v,\mu)) \ \Longrightarrow \
  (0,y^*) \in \bigcup\limits_{h \in \{h \mv \mu \in Df(\xb,\yb)(h,v)\}} \partial f((\xb,\yb);(h,v,\mu)).\]
\end{theorem}
\begin{proof}
  Let $(y^*,-1) \in N_{\epi \vartheta}(\yb,\vartheta(\yb);(v,\mu))$.
  The assumptions imposed on $S$ imply that $S(y)$ is locally not empty-valued (around $\yb$). Hence, we may proceed as in
  \cite[Theorem 10.12]{RoWe98} to obtain $\epi \vartheta = \varphi(\epi f)$ with $\varphi: \R^{n+l+1} \to \R^{l+1}$
  given by $\varphi(x,y,\alpha) = (y,\alpha)$. In order to apply Theorem \ref{The : ConstSetForw} we yet have to show
  that the assumptions on $S$ imply the corresponding assumptions on
  $\Psi(y,\alpha) := \{(x,y,\alpha) \mv (x,y,\alpha) \in \epi f\}$ w.r.t. $\epi \vartheta$.
  Note that if $(y,\alpha) \in \epi \vartheta$ and $x \in S(y)$ we have $\alpha \geq \vartheta(y) = f(x,y)$ and hence
  $\{(x,y,\alpha) \mv (y,\alpha) \in \epi \vartheta, x \in S(y)\} \subset \Psi(y,\alpha)$.

  Consider $(y_k,\alpha_k) \setto{\epi \vartheta} (\yb,\vartheta(\yb))$
  from direction $(v,\mu)$ and thus $y_k \to \yb$ from direction $v$.
  The inner semicompactness of $S$ yields the existence of $\xb$ and a sequence
  $x_k \to \xb$ such that, by passing to a subsequence, we have $x_k \in S(y_k)$
  and hence $(x_k,y_k,\alpha_k) \in \Psi(y_k,\alpha_k)$ with $(x_k,y_k,\alpha_k) \to (\xb,\yb,\vartheta(\yb))$,
  showing the inner semicompactness of $\Psi$ at $(\yb,\vartheta(\yb))$ w.r.t. $\epi \vartheta$ in direction $(v,\mu)$.

  Now fix $\xb \in S(\yb)$. If $S$ is inner semicontinuous at $(\yb,\xb)$ in $v$,
  the inner semicontinuity of $\Psi$ at $((\yb,\vartheta(\yb)),(\xb,\yb,\vartheta(\yb)))$ w.r.t. $\epi \vartheta$ in $(v,\mu)$
  follows from analogous arguments.
  Assuming the inner calmness of $S$, let $\mathcal{V}$ denote the directional neighborhood of $v$
  such that $\xb \in S(y) + L\norm{y - \yb}$ for all $y \in \yb + \mathcal{V}$ and consider a directional
  neighborhood $\mathcal{W}$ of $(v,\mu)$ such that for
  $(y,\alpha) \in \left( (\yb,\vartheta(\yb)) + \mathcal{W} \right) \cap \epi \vartheta$
  we have $y \in \yb + \mathcal{V}$. We obtain that there exists $x \in S(y)$, i.e., $(x,y,\alpha) \in \Psi(y,\alpha)$ such that
  \begin{eqnarray*}
    \norm{(\xb,\yb,\vartheta(\yb)) - (x,y,\alpha)} & \leq & \norm{\xb - x} + \norm{(\yb,\vartheta(\yb)) - (y,\alpha)} \\
    & \leq & L\norm{y - \yb} + \norm{(\yb,\vartheta(\yb)) - (y,\alpha)} \leq (L+1) \norm{(\yb,\vartheta(\yb)) - (y,\alpha)},
  \end{eqnarray*}
  showing the inner calmness of $\Psi$ at $((\yb,\vartheta(\yb)),(\xb,\yb,\vartheta(\yb)))$ w.r.t. $\epi \vartheta$ in $(v,\mu)$.

  Taking into account the differentiability of $\varphi$,
  Theorem \ref{The : ConstSetForw} now yields all statements of the theorem.
\end{proof}

\section{Calculus for directional coderivatives}
In the first part of this section we present two basic calculus rules,
namely the chain rule and the sum rule for directional limiting coderivatives.
In fact, having proved one of them, the other one can be derived relatively easily on the basis of the first one,
similarly like in the case of standard limiting coderivatives.
Here we follow essentially the pattern from \cite{RoWe98}.
Thereafter we present a ``scalarization'' formula which may facilitate
the computation of coderivatives of single-valued Lipschitz continuous mappings.

Consider first the mappings $S_{1}: \mathbb{R}^{n}\rightrightarrows \mathbb{R}^{m},
S_{2}:\mathbb{R}^{m}\rightrightarrows \mathbb{R}^{s}$ and associate with them
the ``intermediate'' multifunction $\Xi: \mathbb{R}^{n} \times \mathbb{R}^{s} \rightrightarrows \mathbb{R}^{m}$ defined by
\[
\Xi(x,u)=\{w\in S_{1}(x)|u \in S_{2}(w)\}.
\]
\begin{theorem}[Directional coderivative chain rule] \label{The : ChainRuleCoder}
Suppose $S=S_{2}\circ S_{1}$ for osc mappings $S_{1},S_{2}$. Let $\bar{x} \in \dom S, \bar{u} \in S(\bar{x})$
and $(h,l)\in \mathbb{R}^{n} \times \mathbb{R}^{s}$  be two given directions. Assume that
\begin{enumerate}
 \item [(a)]
 there is a directional neighborhood $\mathcal{U}$ of $(h,l)$ such that $\Xi((\bar{x},\bar{u})+\mathcal{U})$ is bounded;
 \item [(b)]
 the mapping
 \begin{equation}\label{eq-60}
 F(x,w,u):= \left [
 \begin{array}{l}
 \Gr S_{1}-(x,w)\\
 \Gr S_{2}-(w,u)
 \end{array}
 \right ]
\end{equation}
 is metrically subregular at $(\bar{x},w,\bar{u},0,0)$ for all $w \in \Xi(\bar{x},\bar{u})$
 in  directions $(h,k,l)$ with $k$ such that $(h,k)\in T_{\Gr S_{1}}(\bar{x},w), (k,l)\in T_{\Gr S_{2}}(w,\bar{u})$,
 and in directions $(0,k,0)$ with $k \in \mathbb{S}$ such that $(0,k)\in T_{\Gr S_{1}}(\bar{x},w), (k,0)\in T_{\Gr S_{2}}(w,\bar{u})$;
\end{enumerate}
Then one has
 \begin{equation}\label{incl-1}
\begin{split}
 D^{*}S((\bar{x},\bar{u}); (h,l))\subset
 \bigcup\limits_{ \tilde{w}\in \Xi(\bar{x},\bar{u}) }   &
 \Big( \bigcup\limits_{\scriptstyle k \in \{\xi \in DS_{1}(\bar{x},\tilde{w})(h) | \atop
 \scriptstyle l \in DS_{2}(\tilde{w},\bar{u})(\xi)\} }
  D^{*}S_{1}((\bar{x},\tilde{w}); (h,k))\circ D^{*}S_{2}((\tilde{w},\bar{u});(k,l)) \\
 \cup & \bigcup\limits_{\scriptstyle k \in \{\xi \in \mathbb{S} | DS_{1}(\bar{x},\tilde{w})(0),  \atop
 \scriptstyle 0 \in DS_{2}(\tilde{w},\bar{u})(\xi)\} }
  D^{*}S_{1}((\bar{x},\tilde{w}); (0,k))\circ D^{*}S_{2}((\tilde{w},\bar{u});(k,0)) \Big).
  \end{split}
 \end{equation}
\end{theorem}

\proof
Following the proof idea of \cite[Thm.10.37]{RoWe98} one has that $\Gr S = G(C)$
with $G:(x,w,u)\mapsto (x,u)$ and $C=H^{-1}(D)$, where $H:(x,w,u)\mapsto (x,w,w,u)$ and $D=\Gr S_{1} \times \Gr S_{2}$.

To compute an estimate of $N_{\Gr S}((\bar{x},\bar{u});(h,l))$, we invoke first Theorem 3.2,
which is possible thanks to condition (a), see also Remark \ref{Rem : ConstSetForw}. We obtain that
  \begin{equation}\label{eq-101}
 \begin{split}
 N_{\Gr S}((\bar{x},\bar{u});(h,l)) \subset
 \bigcup\limits_{ \tilde{w}\in \Xi(\bar{x},\bar{u})}&
 \Big(\bigcup\limits_{k \in \{\xi |(h,\xi,l)\in T_{C}(\bar{x},\tilde{w},\bar{u})\}}
 \{(y^{*}_{1}, y^{*}_{2})|(y^{*}_{1},0,y^{*}_{2})\in N_{C}((\bar{x},\tilde{w},\bar{u});(h,k,l))\}\\
 \cup & \bigcup\limits_{k \in \{\xi \in \mathbb{S}|(0,\xi,0)\in T_{C}(\bar{x},\tilde{w},\bar{u})\}}
 \{(y^{*}_{1}, y^{*}_{2})|(y^{*}_{1},0,y^{*}_{2})\in N_{C}((\bar{x},\tilde{w},\bar{u});(0,k,0))\} \Big).
 \end{split}
  \end{equation}

  Next we compute $N_{C}((\bar{x},\tilde{w},\bar{u});(h,k,l))$ via Theorem \ref{The : ConstSet}.
  Thanks to condition (b) and the calmness of $H$ one has
  \begin{equation}\label{eq-102}
  \begin{split}
 &  N_{C}((\bar{x},\tilde{w}, \bar{u}); (h,k,l)) \subset\\
& \{(a,b+c,d)|(a,b)\in N_{\Gr S_{1}}((\bar{x},\tilde{w});(h,k)),(c,d)\in N_{\Gr S_{2}}((\tilde{w},\bar{u}); (k,l))\}
 \end{split}
  \end{equation}
  and, likewise,
\begin{equation}\label{eq-202}
  \begin{split}
 &  N_{C}((\bar{x},\tilde{w}, \bar{u}); (0,k,0)) \subset\\
& \{(a,b+c,d)|(a,b)\in N_{\Gr S_{1}}((\bar{x},\tilde{w});(0,k)),(c,d)\in N_{\Gr S_{2}}((\tilde{w},\bar{u}); (k,0))\}.
 \end{split}
  \end{equation}
Further we observe that
$$
 T_{C}(\bar{x},\tilde{w}, \bar{u}) \subset \{(h,k,l)|(h,k)\in T_{\Gr S_{1}}  (\bar{x},\tilde{w}),(k,l) \in
 T_{\Gr S_{2}}  (\tilde{w},\bar{u})\},
$$
so that the first union in (\ref{eq-101}) with respect to $k$ can be  taken over the set
\begin{equation}\label{eq-61}
k \in \{\xi | (h,\xi)\in T_{\Gr S_{1}} (\bar{x},\tilde{w}),(\xi, l)\in T_{\Gr S_{2}} (\tilde{w},\bar{u})\}
=: \mathcal{T}_{h:l},
\end{equation}
and the second union in (\ref{eq-101}) with respect to $k$ can be taken over the set
\begin{equation}\label{eq-62}
k \in \{\xi \in\mathbb{S} | (0,\xi )\in T_{\Gr S_{1}} (\bar{x},\tilde{w}),(\xi, 0)\in T_{\Gr S_{2}} (\tilde{w},\bar{u})\}
=: \mathcal{T}^{\mathbb{S}}_{0:0}.
\end{equation}
Using consecutively representations (\ref{eq-61}), (\ref{eq-62}) and inclusions (\ref{eq-102}), (\ref{eq-202}) we obtain that
\begin{eqnarray*}
 \lefteqn{N_{\Gr S}((\bar{x},\bar{u}); (h,l))}&&\\
  &\subset& \bigcup\limits_{\tilde{w}\in \Xi(\bar{x},\bar{u})} \Big\{(y^{*}_{1},y^{*}_{2})~|~ (y^{*}_{1},0,y^{*}_{2})\in
 \bigcup\limits_{k \in \mathcal{T}_{h:l}}
 N_{C}((\bar{x},\tilde{w},\bar{u});(h,k,l)) \cup
  \bigcup\limits_{k \in \mathcal{T}^{\mathbb{S}}_{0:0}}
 N_{C}((\bar{x},\tilde{w},\bar{u});(0,k,0)) \Big\}\\
 &\subset& \bigcup\limits_{\tilde{w} \in \Xi(\bar{x},\bar{u})}
 \Big( \bigcup\limits_{k \in \mathcal{T}_{h:l}}
  \{(y^{*}_{1},y^{*}_{2})| \exists c: (y^{*}_{1},-c)\in N_{\Gr S_{1}}((\bar{x},\tilde{w});(h,k)),
  (c,y^{*}_{2})\in N_{\Gr S_{2}}((\tilde{w},\bar{u}); (k,l))\}\\
  && \Big. \hspace{1cm} \cup
 \bigcup\limits_{k \in \mathcal{T}^{\mathbb{S}}_{0:0}}
  \{(y^{*}_{1},y^{*}_{2})| \exists c: (y^{*}_{1},-c)\in N_{\Gr S_{1}}((\bar{x},\tilde{w}); (0,k)),
  (c,y^{*}_{2})\in N_{\Gr S_{2}}((\tilde{w},\bar{u}); (k,0))\} \Big).
\end{eqnarray*}

It follows that for $u^{*}:=-y^{*}_{2}$ one has
   \[
   \begin{split}
 D^{*}S((\bar{x},\bar{u}); (h,l))(u^{*})\subset
 \bigcup\limits_{ \tilde{w}\in \Xi(\bar{x},\bar{u}) } &
 \Big( \bigcup\limits_{k \in \mathcal{T}_{h:l}}
  D^{*}S_{1}((\bar{x},\tilde{w}); (h,k))\circ D^{*}S_{2}((\tilde{w},\bar{u});(k,l))(u^{*})~ \\
 \cup & \bigcup\limits_{k \in \mathcal{T}^{\mathbb{S}}_{0:0}}
  D^{*}S_{1}((\bar{x},\tilde{w}); (0,k))\circ D^{*}S_{2}((\tilde{w},\bar{u});(k,0))(u^{*}) \Big),
  \end{split}
 \]
  and the proof is complete.
\endproof

Let us comment on assumption (b) which is, admittedly, not easy to verify in general.
Following  Proposition \ref{Pro : ProductSet}, it may be ensured by the next two conditions
(which are, however, more restrictive).
\begin{enumerate}
 \item [(i)]
 For all $\tilde{w} \in \Xi (\bar{x},\bar{u})$ and all directions $k$ such that
 $k \in DS_{1}(\bar{x},\tilde{w})(h), l \in DS_{2}(\tilde{w},\bar{u})(k)$ one has
 \begin{equation}\label{eq-100}
0 \in D^{*}S_{1}((\bar{x},\tilde{w});(h,k))(-\lambda), \
\lambda \in D^{*}S_{2}((\tilde{w},\bar{u});(k,l))(0) \quad
\Longrightarrow \quad \lambda = 0;
 \end{equation}
  \item [(ii)]
  for all $\tilde{w}\in \Xi(\bar{x},\bar{u})$ and all directions $k \in \mathbb{S}$ such that
  $k \in DS_{1}(\bar{x},\tilde{w})(0), 0 \in DS_{2}(\tilde{w},\bar{u})(k)$ one has
  \begin{equation}\label{eq-1101}
0 \in D^{*}S_{1}((\bar{x},\tilde{w});(0,k))(-\lambda), \
\lambda \in D^{*}S_{2}((\tilde{w},\bar{u});(k,0))(0) \quad
\Longrightarrow \quad \lambda = 0.
 \end{equation}
\end{enumerate}
We observe that both conditions (\ref{eq-100}), (\ref{eq-1101}) are automatically fulfilled
provided either $S_{1}$ is metrically regular around $(\bar{x},\tilde{w})$ for $\tilde{w} \in \Xi (\bar{x},\bar{u})$,
or $S_{2}$ has the Aubin property around $(\tilde{w},\bar{u})$ for $\tilde{w} \in \Xi (\bar{x},\bar{u})$.
This complies with the corresponding conditions in \cite[Theorem 10.37]{RoWe98}.
More precisely, one can employ the characterizations of the directional metric regularity and the Aubin property
from \cite[Theorem 1]{GfrKl16}.

Further note that in inclusion (\ref{incl-1}) the union over
$k\in \{\xi \in \mathbb{S}| \xi \in DS_{1}(\bar{x},\tilde{w})(0), 0 \in DS_{2}(\tilde{w},\bar{u})(\xi)\}$
vanishes provided we strengthen assumption (a) by asking that the intermediate mapping $\Xi$
is inner calm at $(\bar{x},\bar{u},\tilde{w})$ (w.r.t. $\Gr S$) in direction $(h,l)$.

On the basis of the above considerations we obtain immediately the following corollaries of Theorem \ref{The : ChainRuleCoder}.
\begin{corollary}\label{Cor : ChR1}
In the framework of Theorem \ref{The : ChainRuleCoder} let $S_{1}$ be single-valued and Lipschitz continuous near $\bar{x}$.
Further suppose that multifunction \eqref{eq-60} is metrically subregular at $(\bar{x},S_{1}(\bar{x}), \bar{u}, 0, 0)$
in directions $(h,k,l)$ with $k \in D S_{1}(\bar{x})(h)$ such that $(k,l)\in T_{\Gr S_{2}} (S_{1}(\bar{x}), \bar{u})$.
 Then inclusion (\ref{incl-1}) attains the form
\begin{equation}\label{eq-300}
D^{*}S((\bar{x},\bar{u}); (h,l))\subset
\bigcup\limits_{\scriptstyle k \in \{\xi \in DS_{1}(\bar{x})(h) | \atop
 \scriptstyle l \in DS_{2}(S_{1}(\bar{x}),\bar{u})(\xi)\} }
  D^{*}S_{1}(\bar{x}; (h,k))\circ D^{*}S_{2}((S_{1}(\bar{x}),\bar{u});(k,l)).
\end{equation}
If, moreover, $S_{1}$ is directionally differentiable at $\bar{x}$, then
\begin{equation}\label{eq-301}
D^{*}S((\bar{x},\bar{u}); (h,l))\subset D^{*}S_{1}(\bar{x}; (h,k))\circ D^{*}S_{2}((S_{1}(\bar{x}),\bar{u});(k,l)),
\end{equation}
where $k=S^{\prime}_{1}(\bar{x};h)$.
\end{corollary}
Note that the single-valuedness and the Lipschitz continuity of $S_1$ are carried over to the intermediate mapping $\Xi$,
yielding the fulfillment of assumption (a) as well as the reduction in the estimate.
It is not difficult to verify that the properties of $S_1$ enable us to simplify the
multifunction \eqref{eq-60} by replacing its first row by $S_{1}(x)- w$.
We make use of this fact in the formulation of Theorem \ref{The : SumRuleCoder}.

\begin{corollary}\label{Cor : ChR2}
In the framework of Theorem \ref{The : ChainRuleCoder} let $S_{2}$ be single-valued and Lipschitz continuous near
every $w\in S_{1}(\bar{x})$.
Further, let assumption (a) be fulfilled. Then inclusion \eqref{incl-1} with evident simplifications holds true.
\end{corollary}
The assumptions imposed on $S_2$ justify assumption (b), since for a single-valued mapping Lipschitz continuity and
the Aubin property coincide.

Corollaries \ref{Cor : ChR1} and \ref{Cor : ChR2} represent our main tool in the proof of the next statement.
We consider there mappings $S_{i}:\mathbb{R}^{n} \rightrightarrows \mathbb{R}^{m}, i=1,2,\ldots,p$,
and associate with them the multifunction $\Xi: \mathbb{R}^{n} \times \mathbb{R}^{m} \rightrightarrows (\mathbb{R}^{m})^{p}$
defined by
\[
\Xi(x,u)
= \{w=(w_{1},w_{2},\ldots,w_{p}) \in (\mathbb{R}^{m})^{p} |  w_{i} \in S_{i}(x), \sum w_{i}=u\}.
\]
\begin{theorem}[Directional coderivative sum rule] \label{The : SumRuleCoder}
Suppose $S=S_{1}+S_{2}+ \ldots + S_{p}$ for osc mappings $S_{i}:\mathbb{R}^{n} \rightrightarrows \mathbb{R}^{m} $.
Let $\bar{x}\in \dom S, \bar{u} \in S(\bar{x})$ and $(h,l)\in \mathbb{R}^{n} \times \mathbb{R}^{m}$ be a given pair of directions.
Further assume that
\begin{enumerate}
 \item [(a)]
 there is a directional neighborhood $\mathcal{U}$ of $(h,l)$ such that $\Xi((\bar{x},\bar{u})+\mathcal{U})$ is bounded;
 \item [(b)]
 the mapping
\end{enumerate}
\[
F (x,v,w)=
\left[ \begin{split}
& (x,\ldots,x) - (v_{1},\ldots,v_{p})\\
& \Gr S_{1} - (v_{1},w_{1})\\
& \vdots\\
& \Gr S_{p} - (v_{p},w_{p})
\end{split}\right]
\]
is metrically subregular at $(\bar{x},(\bar{x}, \ldots, \bar{x}), \tilde{w}, 0,0,\ldots, 0)$
for all vectors $\tilde{w} \in \Xi(\bar{x},\bar{u})$ in all directions $(h, (h, \ldots, h),k_{1}, \ldots, k_{p})$ such that
\begin{center}
  $k_{i} \in DS_{i}(\bar{x})(h), i=1,\ldots, p, \, k_{1} + \ldots + k_p=l$,
\end{center}
and in all directions $(0, (0, \ldots, 0),k_{1}, \ldots, k_{p})$ such that
\begin{center}
 $k_{i}\in DS_{i}(\bar{x})(0), i=1,\ldots, p, \, k_{1} + \ldots + k_p=0$ and $k := (k_{1}, \ldots, k_{p}) \in \mathbb{S}$.
\end{center}
Then one has
\begin{equation}\label{eq-2101}
   \begin{split}
D^{*}S((\bar{x},\bar{u}); (h,l))\subset
 & \bigcup\limits_{ \tilde{w} \in \Xi(\bar{x},\bar{u}) }
 \Big( \bigcup\limits_{\scriptstyle  k_{i}\in DS_{i}(\bar{x},\tilde{w}_{i})(h), \atop
 \scriptstyle  \sum k_{i}= l }
  D^{*}S_{1}((\bar{x},\tilde{w}_{1}); (h,k_{1}))+ \ldots +   D^{*}S_{p}((\bar{x}, \tilde{w}_{p}); (h,k_{p})) \\
 &  \hspace{1.3cm} \cup \bigcup\limits_{\scriptstyle k_{i} \in DS_{i}(\bar{x},\tilde{w}_{i})(0), \atop
 \scriptstyle \sum k_{i}=0, k \in \mathbb{S} }
  D^{*}S_{1}((\bar{x},\tilde{w}_{1}); (0,k_{1}))+ \ldots +   D^{*}S_{p}((\bar{x}, \tilde{w}_{p}); (0,k_{p})) \Big).
  \end{split}
 \end{equation}
\end{theorem}

\proof
Following \cite[Theorem 10.41]{RoWe98}, we observe that $S=F_{2}\circ S_{o} \circ F_{1}$ for $S_{o}:
(x_{1},x_{2},\ldots,x_{p}) \mapsto S_{1}(x_{1}) \times \ldots \times S_{p}(x_{p}),
F_{1}: x \mapsto (x, \ldots, x)$($p$ copies) and $F_{2}:(u_{1}, \ldots, u_{p})\mapsto u_{1}+ \ldots + u_{p}$.
So, it suffices to apply first Corollary \ref{Cor : ChR2} to the composition $F_{2} \circ G$ for $G=S_{o} \circ F_{1}$
(which is possible under condition (a)) and, thereafter, Corollary \ref{Cor : ChR1} to compute the directional limiting
coderivative of $G$. This can be done under assumption (b) and leads directly to formula \eqref{eq-2101}.
\endproof

Similarly as in Theorem \ref{The : ChainRuleCoder}, the union over
$
\{k \in \mathbb{S} \mv k_{1} + \ldots + k_p =0, k_{i} \in D S_{i}(\bar{x},\tilde{w}_{i}) (0)\}
$
vanishes provided we strengthen assumption (a) by asking that the intermediate mapping $\Xi$
is inner calm at $(\bar{x},\bar{u},\tilde{w})$ (w.r.t. $\Gr S$) in direction $(h,l)$.

Condition (b) can be ensured by the  assumptions (i),(ii) below, which follow from  implications (\ref{eq-100}), (\ref{eq-1101})
 applied to the composition $S_{o} \circ F_{1}$. They attain the form:
 \begin{enumerate}
  \item [(i)]
  For all $\tilde{w} \in \Xi (\bar{x},\bar{u})$ and all directions $k_{1}, \ldots, k_{p}$ such that
  $k_{i}\in DS_{i}(\bar{x},\tilde{w}_{i})(h), i=1, \ldots, p, k_{1}+\ldots+k_p=l $, one has
  \[
  v^{*}_{i}\in D^{*}S_{i}((\bar{x},\tilde{w}_{i});(h,k_{i}))(0), \ \sum v^{*}_{i}=0 \quad
  \Longrightarrow \quad v^{*}_{i}=0 \mbox{ for } i=1,\ldots,p.
  \]
   \item [(ii)]
  For all $\tilde{w} \in \Xi (\bar{x},\bar{u})$ and all directions $k_{1}, \ldots, k_{p}$ such that
  $k_{i}\in DS_{i}(\bar{x},\tilde{w}_{i})(0), i=1, \ldots, p, k_{1}+\ldots+k_p=0, k \in \mathbb{S}$, one has
  \[
 v^{*}_{i}\in D^{*}S_{i}((\bar{x},\tilde{w}_{i});(0,k_{i}))(0), \ \sum v^{*}_{i}=0 \quad
  \Longrightarrow \quad v^{*}_{i}=0 \mbox{ for } i=1,\ldots,p.
  \]
 \end{enumerate}
 These conditions represent a directional version of \cite[Theorem 10.41, condition (b)]{RoWe98}.

\begin{corollary}
In the setting of Theorem \ref{The : SumRuleCoder} assume that $p=2$ and $S_{1}$ is single-valued and
Lipschitz continuous near $\bar{x}$ and directionally differentiable at $\bar{x}$.
Then all assumptions of Theorem \ref{The : SumRuleCoder} are fulfilled and
\begin{equation}\label{eq-103}
D^{*}S((\bar{x},\bar{u});(h,l)) \subset
D^{*}S_{1}(\bar{x};(h,k)) +
D^{*}S_{2}((\bar{x},\bar{u}-S_{1}(\bar{x}));(h,l-k)),
\end{equation}
where $k=S_{1}(\bar{x};h)$.
\end{corollary}

\proof
First we observe that $\Xi(\bar{x},\bar{u})=\{(S_{1}(\bar{x}),\bar{u}-S_{1}(\bar{x})) \}$
and all assumptions of Theorem \ref{The : SumRuleCoder} are fulfilled thanks to the assumed properties of $S_{1}$.
Formula (\ref{eq-103}) follows directly from inclusion (\ref{eq-2101}).
\endproof

Note that inclusion (\ref{eq-103}) becomes equality provided $S_{1}$ is continuously differentiable near $\bar{x}$,
cf. \cite[formula(2.4)]{GO3}.

The next statement is a directional version of the useful scalarization formula for single-valued Lipschitz continuous mappings,
cf. \cite[Proposition 9.24]{RoWe98} and \cite[Theorem 3.28]{Mo06a}.

\begin{proposition} \label{Pro : Scal}
  Consider a single-valued continuous mapping $\varphi: \R^n \to \R^m$
  which is also Lipschitz continuous near $\xb$ in direction $u$. Then for any $y^*,v \in \R^m$ one has
  \[D^*\varphi(\xb;(u,v))(y^*) = \partial \skalp{y^*,\varphi} (\xb;(u,\skalp{y^*,v})).\]
\end{proposition}
\begin{proof}
  Let $x^* \in D^*\varphi(\xb;(u,v))(y^*)$, i.e., $(x^*,y^*) \in N_{\Gr \varphi}((\xb,\varphi(\xb));(u,v))$ and consider sequences
  $(t_k) \downarrow 0$, $(u_k,v_k) \to (u,v)$, $(x_k^*,y_k^*) \to (x^*,y^*)$ with
  $(x_k^*,y_k^*) \in \hat N_{\Gr \varphi}((\xb,\varphi(\xb)) + t_k(u_k,v_k))$.
  Thus, we have $v_k = (\varphi(\xb + t_k u_k) - \varphi(\xb))/t_k$ and, taking into account
  that $\varphi$ is Lipschitz continuous near $\xb + t_k u_k$ and applying \cite[Proposition 9.24(b)]{RoWe98}, we obtain
  \[x_k^* \in \hat D^*\varphi(\xb + t_k u_k)(y_k^*) = \hat \partial \skalp{y_k^*,\varphi}(\xb + t_k u_k) =
    \hat \partial [\skalp{y^*,\varphi} + \skalp{y_k^* - y^*,\varphi}](\xb + t_k u_k).\]
  The fuzzy sum rule \cite[Theorem 2.33]{Mo06a} yields the existence of
  $x_1,x_2$ with $\norm{x_i - (\xb + t_k u_k)} \leq t_k^2$ for $i=1,2$ such that
  \[x_k^* \in \hat \partial \skalp{y^*,\varphi}(x_1) +
  \hat \partial \langle y^{*}_{k}-y^{*}, \varphi \rangle (x_{2})\subset  \hat \partial \skalp{y^*,\varphi}(x_1)+
   (\norm{y_k^* - y^*} \kappa + t_k^2 )\B, \]
  where $\kappa$ denotes the Lipschitz constant.
  Consequently, there exists $\xi_k^* \in \hat \partial \skalp{y^*,\varphi}(x_1)$ with $\xi_k^* \to x^*$
  and similarly as in the proof of Theorem \ref{The : ConstSet} one can easily
  show that $\tilde u_k := (x_1-\xb)/t_k \to u$ and Lipschitz continuity of $\varphi$
  then implies also $(\varphi(x_1) - \varphi(\xb))/t_k \to v$. Thus, we conclude
  \[(\xi_k^*,-1) \in \hat N_{\epi \skalp{y^*,\varphi}}
  (\xb + t_k \tilde u_k, \skalp{y^*,\varphi}(\xb) + t_k \skalp{y^*,(\varphi(\xb + t_k \tilde u_k) - \varphi(\xb))/t_k}),\]
  showing $x^* \in \partial \skalp{y^*,\varphi} (\xb;(u,\skalp{y^*,v}))$.

  The reverse inclusion follows easily from Theorem \ref{The : ChainRule}
  since $\skalp{y^*,\varphi} = g \circ \varphi$ for linear function $g = \skalp{y^*,\cdot}$.
  Note that $\partial^{\infty}g(\varphi(\xb))=\{0\}$ implies metric subregularity of $F(x,\alpha) = \epi g - (\varphi(x),\alpha)$.
\end{proof}

An equivalent formulation of the following useful result was also proven in \cite[Corollary 5.9]{LoWaYa17}.

\begin{corollary} \label{Cor : SingDifLipsch}
  For an lsc function $f:\R^n \to \bar \R$ that is also Lipschitz continuous near $\xb$ in direction $h$
  one has $\partial^{\infty}f(\xb;(h,\nu))\subset \{0\}$ for all $\nu \in \R$.
\end{corollary}
\begin{proof}
  Indeed, $\partial^{\infty}f(\xb;(h,\nu)) \subset D^*f(\xb;(h,\nu))(0) = \partial \skalp{0,f} (\xb;(h,\skalp{0,\nu})) = \{0\}$.
\end{proof}

\section{Applications}

In this section we apply some of the above presented calculus rules to several problems of variational analysis,
where directional notions can be advantageously utilized.

\subsection{First-order sufficient conditions for directional metric regularity and subregularity of feasibility mappings}

Consider a mapping of the form $F(x)= \Omega - \varphi(x)$ which arises in qualification conditions throughout the whole paper.
The next result (announced already in Section 2)
extends the results from \cite[Theorem 1, Corollary 1]{GfrKl16},
where $\varphi$ is assumed to be continuously differentiable.

\begin{theorem}
  Let the multifunction $F : \R^n \rightrightarrows \R^m$ be given by $F(x)= \Omega - \varphi(x)$,
  where $\varphi: \R^n \to \R^m$ is continuous and $\Omega \subset \R^m$ is a closed set.
  Further let $(\xb,0) \in \Gr F$,  $(u,v) \in \R^{n}\times \R^{m}$ be given and assume
  that $\varphi$ is calm at $\xb$ in direction $u$. Then
  \begin{enumerate}
   \item $F$ is metrically subregular at $(\xb,0)$ in direction $u$ provided
   for all $w \in D\varphi(\xb)(u) \cap T_{\Omega}(\varphi(\xb))$ one has the implication
   \[0 \in D^*\varphi(\xb;(u,w))(\lambda), \,
   \lambda \in N_{\Omega}(\varphi(\xb);w) \ \Longrightarrow \ \lambda = 0.\]
   \item $F$ is metrically regular at $(\xb,0)$ in direction $(u,v)$ (cf. \cite[Definition 1]{Gfr13a})
   provided for all $w \in D\varphi(\xb)(u)$ with $w + v \in T_{\Omega}(\varphi(\xb))$ one has the implication
   \[0 \in D^*\varphi(\xb;(u,w))(\lambda), \, \lambda \in N_{\Omega}(\varphi(\xb);v+w) \ \Longrightarrow \ \lambda = 0.\]
  \end{enumerate}
\end{theorem}

The proof of the second statement is based on \cite[Theorem 5]{Gfr13a} which provides equivalent
characterizations of directional metric
regularity and in finite dimensional spaces one of them states that $F$ is metrically regular at $(\xb,0)$
in direction $(u,v)$ if and only if $0 \in D^*F((\xb,0);(u,v))(\lambda)$ implies $\lambda = 0$.
The first statement then follows from the fact that metric regularity of $F$ at $(\xb,0)$ in direction $(u,0)$
implies metric subregularity of $F$ at $(\xb,0)$ in direction $u$.
Thus, it suffices to show the following lemma.

\begin{lemma}
  Under the settings and assumptions of the previous theorem we have
  \[D^*F((\xb,0);(u,v))(-\lambda) \subset
  \begin{cases}
  \bigcup_{w \in D\varphi(\xb)(u), \atop w+v \in T_{\Omega}(\varphi(\xb)))}
  D^*\varphi(\xb;(u,w))(\lambda) & \textrm{ if } \lambda \in N_{\Omega}(\varphi(\xb);v+w),\\
  \emptyset & \textrm{ otherwise.}
  \end{cases}\]
\end{lemma}
\begin{proof}
The assumed calmness of $\varphi$ implies that the intermediate mapping $\Xi(x,y)=\{y+\varphi(x),-\varphi(x)\}$
is inner calm at $(\bar{x},0,\varphi(\bar{x}),-\varphi(\bar{x}))$ in direction $(u,v)$.
On the other hand, denoting by $G$ the mapping that to each $x$ assigns set $\Omega$, it is clear that $G$ has
the Aubin property and we may apply the sum rule for coderivatives, Theorem \ref{The : SumRuleCoder}.
The statement of the lemma thus follows from the fact that for some $w$ we obtain
\begin{eqnarray*}
  D^*G((\xb,\varphi(\xb));(u,v+w))(-\lambda) & = &
  \{\xi \mv (\xi,\lambda) \in N_{\R^n \times \Omega}((\xb,\varphi(\xb));(u,v+w))\} \\
  & = & \begin{cases}
  0 & \textrm{ if } \lambda \in N_{\Omega}(\varphi(\xb);v+w),\\
  \emptyset & \textrm{ otherwise.}
  \end{cases}
\end{eqnarray*}
\end{proof}

\subsection{Subtransversality of set systems}
Consider the collection  of closed sets $C_{1}, C_{2}, \ldots, C_{l} $ from $\mathbb{R}^{n}$
and a point $\bar{x} \in C:= \bigcap\limits^{l}_{i=1} C_{i}$.
By the definition (cf., e.g., \cite[Definition 1(ii)]{AK17}),
these sets are {\em subtransversal} at $\bar{x}$ provided there exist a neighborhood $U$ of $\bar{x}$
and a constant $L>0$ such that the metric inequality
\[
d_{C}(x)\leq L \sum\limits^{l}_{i=1} d_{C_{i}}(x)
\]
holds for all $x \in U$. This is, on the other hand, equivalent with the calmness of the {\em perturbation} mapping
\[
S(p_{1},\ldots, p_{l})=\{x|p_{i}+ x \in C_{i}, i=1,2,\ldots, l\}
\]
 at $(0,\ldots,0,\bar{x})$, cf. \cite[Section 3]{IoOut08}.
 A straightforward application of \cite[Theorem 3.8]{GO3} yields the following result.
\begin{theorem}\label{Appl2}
Assume that there do not exist nonzero vectors $u \in \mathbb{R}^{n}, v^{*}_{i}\in\mathbb{R}^{n}, i=1,2,\ldots,l$, such that
\[
u \in \bigcap\limits^{l}_{i=1}T_{C_{i}} (\bar{x}), \quad
0 = \sum\limits^{l}_{i=1} v^{*}_{i}, \quad
v^{*}_{i} \in N_{C_{i}} (\bar{x},u), i=1,2,\ldots,l.
\]
Then collection $\{C_{1}, C_{2}, \ldots, C_{l}\}$ is subtransversal at $\bar{x}$.
\end{theorem}
Very often the sets $C_{i}$ correspond to various constraint systems and can be described as
\begin{equation}\label{eq-6.2}
C_{i}=\varphi^{-1}_{i}(Q_{i})
\end{equation}
with $Q_{i}\subset \mathbb{R}^{m_{i}}$ being closed and $\varphi_{i}:\mathbb{R}^{n} \rightarrow \mathbb{R}^{m}$
being Lipschitz continuous near $\bar{x}$. As a simple consequence of Theorem \ref{Appl2}
we obtain a condition ensuring the subtransversality of a collection  of pre-images.
\begin{corollary} \label{Cor : Transv}
In the setting of Theorem \ref{Appl2} assume that the sets $C_{i}$ are given via (\ref{eq-6.2}) where,
in addition to the posed assumptions, functions $\varphi_{i}$ are directionally differentiable at $\bar{x}$.
Further assume that the mappings
\[
F_{i}(x)= Q_{i} - \varphi_{i}(x), i=1,2,\ldots,l,
\]
are metrically subregular at $\bar{x}$. Finally suppose that there do not exist nonzero vectors
$u\in \mathbb{R}^{n}, v^{*}_{i}\in \mathbb{R}^{n}, i=1,2,\ldots,l,$ such that
\begin{align}
& u \in \{h \in \mathbb{R}^{n} | \varphi^{\prime}_{i}(\bar{x};h)\in T_{Q_{i}}(\varphi(\bar{x})),i=1,2,\ldots,l\},
\label{eq-6.3}\\
& 0 = \sum\limits^{l}_{i=1} v^{*}_{i},
 \label{eq-6.4}\\
& v^{*}_{i} \in D^{*}\varphi_{i} (\bar{x};u)N_{Q_{i}}(\varphi_{i}(\bar{x}); \varphi^{\prime}_{i}(\bar{x};u)), i=1,2,\ldots,l.
 \label{eq-6.5}
\end{align}
Then collection  $\{C_{1},C_{2},\ldots, C_{l}\}$ is subtransversal at $\bar{x}$.
\end{corollary}
The proof follows easily from Theorem \ref{Appl2}, Theorem \ref{The : ConstSet}
and the fact that the set on the right-hand side of (\ref{eq-6.3}) amounts exactly to $\bigcap\limits^{l}_{i=1}T_{C_{i}} (\bar{x})$.
\begin{example}
Let $n=2, l=2,$
\begin{equation}\label{eq-6.6}
\begin{split}
& C_{1} = \{x \times \mathbb{R}^{2}| -x_{1}-x^{2}_{1}+x_{2}\leq 0, -x_{1}-x^{2}_{1}-x_{2}\leq 0\},\\
& C_{2}= \left\{ x \times \mathbb{R}^{2} \left|\left(
\begin{array}{r}
x_{2}\\
-x_{1}
\end{array}
\right)
\in \Gr N_{\mathbb{R}_{+}} \right.\right\},
\end{split}
\end{equation}
and $\bar{x}=(0,0)$. It is easy to verify that all assumptions of Corollary \ref{Cor : Transv}
are fulfilled and the only direction $u$ satisfying (\ref{eq-6.3}) is the direction $\mathbb{R} _{+}(1,0)$.
Clearly, with $\varphi_{1}$ and $Q_{1}$ given in (\ref{eq-6.6}) one has
\[
\nabla \varphi_{1}(\bar{x})^{T} N_{Q_{1}}(\varphi_{1}(\bar{x});\nabla \varphi_{1}(\bar{x})u)=
\left[ \begin{array}{rr}
-1 & -1\\
1 & -1
\end{array} \right ]
N_{\mathbb{R}^{2}_{+}}((0,0);(-1,-1))=\{(0,0)\}.
\]
Consequently, there does not exist any nonzero pair $v^{*}_{1}, v^{*}_{2}$ satisfying conditions
(\ref{eq-6.4}), (\ref{eq-6.5}) and so  collection  $\{C_{1},C_{2}\}$ is subtransversal at $\bar{x}$.

Note that we are not able to detect this property via the (stronger) Aubin property of $S$ because, with $\bar{p}=0$,
\[
\begin{split}
 D^{*}S(\bar{p},\bar{x})(0)= &
\left\{(a^{*},b^{*})\in \mathbb{R}^{2} \times \mathbb{R}^{2} \left| a^{*} \in
\left[ \begin{array}{rr}
-1 & -1\\
1 & -1
\end{array} \right ] \mathbb{R}^{2}_{+},
b^{*} \in
\left[ \begin{array}{rr}
0 & -1\\
1 & 0
\end{array} \right ] N_{\Gr N_{\mathbb{R}_{+}}}(0,0),\right.\right.\\
& a^{*}+b^{*}=0 \}.
\end{split}
\]
Thus, since the vectors $a^{*}=(-1,0)$ and $b^{*}=(1,0)$ belong to $D^{*}S(\bar{p},\bar{x})(0)$,
we conclude from the Mordukhovich criterion that $S$ does not possess the Aubin property around
$(\bar{p},\bar{x})$.\hfill $\triangle$
\end{example}

\subsection{Aubin property of implicitly given mappings}

By combination of \cite[Theorem 4.4]{GO3} with Proposition \ref{Pro : Scal}
one obtains a sufficient condition for the Aubin property for a class of implicitly defined multifunctions.
Let the function $M: \mathbb{R}^{l}\times  \mathbb{R}^{n} \rightarrow \mathbb{R}^{m}$
be Lipschitz continuous near the reference point $(\bar{p},\bar{x}) \in \mathbb{R}^{l}\times \mathbb{R}^{n}$
satisfying $M(\bar{p},\bar{x})=0$ and consider the solution mapping
\[
S(p):= \{x \in \mathbb{R}^{n} | M(p,x)=0\}.
\]
\begin{theorem}\label{Appl3}
Assume that $M$ is directionally differentiable at $(\bar{p},\bar{x})$ and
\begin{enumerate}
 \item [(i)]
 $\{u\in \mathbb{R}^{n} | M^{\prime}((\bar{p},\bar{x});(v,u))=0\} \neq \emptyset \mbox{ for all } v \in \mathbb{R}^{l}$;
 \item [(ii)]
 $M$ is metrically subregular at $(\bar{p},\bar{x},0)$;
 \item [(iii)]
 For every nonzero $(v,u)\in \mathbb{R}^{l} \times \mathbb{R}^{n}$ such that $M^{\prime}((\bar{p},\bar{x});(v,u))=0 $
 one has the implication
 \[
 \left[ \begin{array}{l}
q^{*}\\
0
\end{array} \right ] \in \partial \langle y^{*},M  \rangle ((\bar{p},\bar{x});(v,u,0))
\Rightarrow q^{*}=0.
 \]
 \end{enumerate}
 Then $S$ has the Aubin property around $(\bar{p},\bar{x})$. \\
 This statement remains valid if conditions (ii), (iii) are replaced by the (stronger) implication
 \begin{equation}\label{eq-6.7}
 \left[ \begin{array}{l}
q^{*}\\
0
\end{array} \right ] \in \partial \langle y^{*},M  \rangle ((\bar{p},\bar{x});(v,u,0))
\Rightarrow y^{*}=0.
 \end{equation}
 \end{theorem}
 This result can be applied to parameterized nonlinear complementarity problems (NCPs) governed by the equation
\begin{equation}\label{eq-6.8}
0 = M(p,x):= \min \{G(p,x), H(p,x)\},
\end{equation}
where functions $G,H: \mathbb{R}^{l} \times \mathbb{R}^{n} \rightarrow \mathbb{R}^{n}$ are Lipschitz continuous
near $(\bar{p},\bar{x})$,
directionally differentiable at $(\bar{p},\bar{x})$ and the ``minimum'' is taken componentwise.
As always in the treatment of finite-dimensional NCPs we introduce for $(p,x)\in \Gr S$ the index sets
\[
\begin{split}
& I_{G}(p,x):= \{i \in \{1,2,\ldots,n\} | G_{i}(p,x)=0, H_{i}(p,x)> 0\}\\
& I_{H}(p,x):= \{i \in \{1,2,\ldots,n\} | G_{i}(p,x)>0, H_{i}(p,x)= 0\}\\
& I_{0}(p,x):= \{i \in \{1,2,\ldots,n\} | G_{i}(p,x)= H_{i}(p,x)= 0\},
\end{split}
\]
which create a partition of $\{1,2,\ldots,n\}$. To be able to apply Theorem \ref{Appl3} to $M$ given by (\ref{eq-6.8})
we observe that $M^{\prime}((\bar{p},\bar{x});(v,u))$ (for general directions $(v,u)$) amounts to the vector $b$ such that
\[
b_{i} =\left\langle
\begin{split}
& G^{\prime}_{i}((\bar{p},\bar{x});(v,u)) \mbox{ for } i \in I_{G}(\bar{p},\bar{x})\cup I_{0G}(\bar{p},\bar{x})\\
& H^{\prime}_{i}((\bar{p},\bar{x});(v,u)) \mbox{ for } i \in I_{H}(\bar{p},\bar{x})\cup I_{0H}(\bar{p},\bar{x})\\
& G^{\prime}_{i}((\bar{p},\bar{x});(v,u))=
H^{\prime}_{i}((\bar{p},\bar{x});(v,u))
\mbox{ otherwise, }
\end{split} \right.
\]
where $I_{0G}(\bar{p},\bar{x}):= \{j \in I_{0}(\bar{p},\bar{x})|G^{\prime}_{j}((\bar{p},\bar{x});(v,u))<
H^{\prime}_{j}((\bar{p},\bar{x});(v,u))\}$ and $I_{0H}(\bar{p},\bar{x}):= \{j \in I_{0}(\bar{p},\bar{x})|H^{\prime}_{j}((\bar{p},\bar{x});(v,u))<
G^{\prime}_{j}((\bar{p},\bar{x});(v,u))\}$.

Furthermore, by virtue of Propositions \ref{Pro : MaxSubdif}, \ref{Pro : MinSubdif}
and the definition of the directional limiting subdifferential, we obtain
\[
\begin{split}
& \partial \langle y^{*},M \rangle ((\bar{p},\bar{x});(v,u,0))\subset\\
& \sum\limits_{ i \in I_{G}(\bar{p},\bar{x})\cup I_{0G}(\bar{p},\bar{x})}\Phi_{i}(y^{*}_{i}) +
\sum\limits_{ i \in I_{H}(\bar{p},\bar{x})\cup I_{0H}(\bar{p},\bar{x})}\Psi_{i}(y^{*}_{i}) +
\sum\limits_{ i \in I_{0}(\bar{p},\bar{x})\setminus (I_{0G}(\bar{p},\bar{x}) \cup I_{0H}(\bar{p},\bar{x}))}\Theta_{i}(y^{*}_{i}),
\end{split}
\]
where the multifunctions $\Phi_{i}, \Psi_{i}, \Theta_{i}, i=1,2,\ldots,n, \mbox{ map } \mathbb{R}$
to (subsets of) $\mathbb{R}^{n}$ and, for $a \in \mathbb{R}$, are defined by
\[
\begin{split}
& \Phi_{i}(a) =\left\langle
\begin{split}
&  a \partial G_{i}((\bar{p},\bar{x});(v,u,0)) \mbox{ if } a > 0\\
&  |a| \partial (-G_{i})((\bar{p},\bar{x});(v,u,0)) \mbox{ if } a < 0\\
&0 \mbox{ otherwise, }
\end{split} \right.\\
& \Psi_{i}(a) =\left\langle
\begin{split}
&  a \partial H_{i}((\bar{p},\bar{x});(v,u,0)) \mbox{ if } a > 0\\
&  |a| \partial (-H_{i})((\bar{p},\bar{x});(v,u,0)) \mbox{ if } a < 0\\
&0 \mbox{ otherwise, }
\end{split} \right.\\
&\Theta_{i}(a) =\left\langle
\begin{split}
&  a [\partial G_{i}((\bar{p},\bar{x});(v,u,0)) \cup \partial H_{i}((\bar{p},\bar{x});(v,u,0))] \mbox{ if } a > 0\\
&  |a| \mbox{ conv } [\partial (-G_{i})((\bar{p},\bar{x});(v,u,0)), \partial (-H_{i})((\bar{p},\bar{x});(v,u))] \mbox{ if } a < 0\\
& 0 \mbox{ otherwise. }
\end{split} \right.
\end{split}
\]

The usage of the above formulas is illustrated by the following nonsmooth NCP.
\begin{example}
Consider the NCP governed by equations (\ref{eq-6.8}) with $l=n=1, G(p,x)=x-x^{2}, H(p,x)=|p|-x-x^{2}$
and put $(\bar{p},\bar{x})=(0,0)$. One has
\[
M^{\prime}((\bar{p},\bar{x});(v,u))=\min \{u,|v| -u\}
\]
and so assumption (i) of Theorem \ref{Appl3} is fulfilled. Implication (\ref{eq-6.7}) has to be verified
for directions $(v,u)$ satisfying either $v\neq 0, u=0$ or $v\neq 0, u=|v|$.
In the first case one has $\{1\} \in I_{0G}(\bar{p},\bar{x})$ and so implication (\ref{eq-6.7}) attains the form
\[
\left[ \begin{array}{l}
q^{*}\\
0
\end{array} \right ] = y^{*}
\left[ \begin{array}{l}
0\\
1
\end{array} \right ] \Rightarrow y^{*}=0.
\]
In the second case one has $\{1\} \in I_{0H}(\bar{p},\bar{x})$ and so we have to verify the conditions
\[
\begin{split}
\nexists ~y^{*}> 0 : & \left[ \begin{array}{l}
q^{*}\\
0
\end{array} \right ] \in y^{*}
 \left[ \begin{array}{l}
[-1,1]\\
-1
\end{array} \right ]  \\[1ex]
\nexists~ y^{*}<0 : & \left[ \begin{array}{l}
q^{*}\\
0
\end{array} \right ] \in |y^{*}|
 \left[ \begin{array}{c}
\{-1\}\cup \{1\}\\
-1
\end{array} \right ].
\end{split}
\]
Since both these conditions are fulfilled, the corresponding mapping $S$ has the Aubin property around $(\bar{p},\bar{x})$.

Note that the standard condition ensuring the Aubin property of $S$ via the Mordukhovich criterion is violated.
Indeed, by the calculus from \cite[Chapter 3]{Mo06a} one has
\begin{equation}\label{eq-6.9}
D^{*}M(\bar{p},\bar{x})(y^{*})= \partial \langle y^{*}, M\rangle (\bar{p},\bar{x})\subset
\left \langle
\begin{split}
& y^{*}
\left( \left\{ \left[ \begin{array}{l}
0\\
1
\end{array} \right ]\right\} \cup
\left\{ \left[ \begin{array}{l}
[-1,1]\\
-1
\end{array} \right ]\right\}\right) ~{\rm if }~ y^{*}>0\\
& |y^{*}| ~{\rm conv } ~\left(\left[ \begin{array}{r}
0\\
-1
\end{array} \right ],
\left[ \begin{array}{c}
\{-1\}\cup \{1\}\\
1
\end{array} \right ]\right)~ {\rm if }~ y^{*}<0.
\end{split}
\right.
\end{equation}
One can easily verify that, for instance, the nonzero pair $(q^{*},y^{*})=(0.5,-1)$ satisfies the system
\[
\left[ \begin{array}{l}
q^{*}\\
0
\end{array} \right ] \in \partial \langle y^{*},M\rangle (\bar{p},\bar{x})
\]
and so the estimate (\ref{eq-6.9}) is not precise enough to enable us to detect
the Aubin property of $S$ around $(\bar{p},\bar{x})$ via the Mordukhovich criterion. \hfill$\triangle$
\end{example}

\medskip

\subsection{Improving the standard calculus}

It can easily be seen that all rules presented in Sections 3-5 reduce to their counterparts from the classical generalized differential calculus provided we set the considered directions to be zero. In some cases, however, the classical rules may  even be improved when one employs the appropriate results from this paper. This concerns both the restrictiveness of the imposed assumptions as well as the sharpness of the resulting estimates.

As to the former case, Proposition 6.1 below extends a statement from \cite[Theorem 6.43]{RoWe98}
by a substantial relaxation  of the assumptions.

\begin{proposition}
  Consider a closed set $C \subset \R^n$ and a continuous mapping $\varphi: \R^n \to \R^l$,
  set $Q = \varphi(C)$ and consider $\yb \in Q$.
  Let $\Psi : \R^l \rightrightarrows \R^n$ be given by $\Psi(y) := \varphi^{-1}(y) \cap C$
  and let it be inner semicontinuous at $(\yb,\xb)$ w.r.t. $Q$ for some $\xb \in \Psi(\yb)$.
  Assume further that the set-valued mapping $F: \R^{l+n} \rightrightarrows \R^{2(l+n)}$ given by
  $F(y,x) = \left( \Gr \varphi^{-1} - (y,x) \right) \times \left( (\R^l \times C) - (y,x) \right)$ is
  metrically subregular at $((\yb,\xb),(0,\ldots,0))$. Then
  \begin{equation*}
    N_Q(\yb) \ \subset \ \{y^* \mv D^*\varphi(\xb)(y^*) \cap N_C(\xb) \neq \emptyset\}.
  \end{equation*}
\end{proposition}
The statement follows from Theorem \ref{The : ConstSetForw} for direction $v=0$,
taking also into account Remark \ref{Rem : ConstSetForw}.
Note that the second term in \eqref{eqn : ConstSetForw} is covered by the first one
and no inner calmness assumption is thus needed.

Next we show a possible improvement of two estimates for the limiting normal cones. 

\begin{proposition}
  Given a closed set $Q \subset \R^m$ and  a continuously differentiable function $\varphi: \R^n \to \R^m$,
    consider the set $C := \varphi^{-1}(Q)$.
  Assume  that the set-valued mapping $F: \R^n \rightrightarrows \R^m$ given by $F(x) = Q - \varphi(x)$ is
  metrically subregular at $(\xb,0)$ in every direction $h \in T_C(\xb) \cap \mathbb{S}$. Then
  \[N_C(\xb) \subset \hat N_C(\xb) \cup \bigcup\limits_{h \in T_C(\xb) \cap \mathbb{S}}
  \nabla \varphi(\xb)^T N_Q(\varphi(\xb);\nabla \varphi(\xb) h).\]
\end{proposition}
The proof follows easily from Lemma \ref{lem : LimNorConeViaDir}  and Theorem \ref{The : ConstSet}.
The increased sharpness of this estimate with respect to \cite[Theorem 6.14]{RoWe98}  has been used in sufficient conditions for the Aubin property
of a class of solution mappings in \cite{GO3,GO4}. It stays also behind the application discussed in Subsection 6.3.

A sharper estimate can be obtained in this way also in the case of normal cones to unions. Consider a family of closed sets $C_{i} \subset \mathbb{R}^{n}, i=1,2,\ldots, l$, and a point $\bar{x}\in C=\cup_{i=1}^l C_i$. The standard estimate attains  the form 
 
\begin{equation}\label{eq : UnionSimp}
  N_C(\xb) \subset \bigcup_{i \in I(\xb)} N_{C_i}(\xb),
\end{equation}
where $I(\xb) := \{i =1,\ldots,l \mv \xb \in C_i\}$. This follows, e.g., from Proposition \ref{Prop : Cup} with $h=0$.  
On the basis of Lemma \ref{lem : LimNorConeViaDir}, however, we obtain the estimate
\[N_{C}(\xb) \subset \bigcap_{i \in I(\xb)} \hat N_{C_i}(\xb) \cup \bigcup_{h \in T_C(\xb) \cap \mathbb{S}}
\bigcup_{i \in I(\xb,h)} N_{C_i}(\xb;h),\]
where $I(\xb,h) := \{i \in I(\xb) \mv h \in T_{C_i}(\xb)\}$.
This estimate is tighter then \eqref{eq : UnionSimp}, which is demonstrated in the next example.

\begin{example}[MPCC generating set]
  Let $C_1 = \R_+ \times \{0\}$, $C_2 = \{0\} \times \R_+$ and $\xb=(0,0)$. We have
  \[N_{C_1 \cup C_2}(\xb) = \left( \R_- \times \R_- \right) \cup \left( \{0\} \times \R_+ \right) \cup \left( \R_+ \times \{0\} \right),\]
  $N_{C_1}(\xb) = ( \R_- \times \R )$ and $N_{C_2}(\xb) = ( \R \times \R_- )$. This shows that in this case the standard
  calculus does not provide a tight estimate of the limiting normal cone.

  On the other hand, we obtain $\bigcap_{i \in I(\xb)} \hat N_{C_i}(\xb) = \R_- \times \R_-$,
  $T_{C}(\xb) \cap \mathbb{S} = (C_1 \cup C_2)\cap \mathbb{S} = \{(1,0),(0,1)\}$ and
  \begin{eqnarray*}
    I(\xb,(1,0)) = \{1\}, && N_{C_1}(\xb,(1,0)) = \{0\} \times \R, \\
    I(\xb,(0,1)) = \{2\}, && N_{C_2}(\xb,(0,1)) = \R \times \{0\},
  \end{eqnarray*}
  which yields that the estimate based on the directional calculus is in fact exact. \hfill$\triangle$
\end{example} 

\section*{Conclusion}
The paper contains directional variants of almost all important rules
arising in generalized differential calculus of limiting normal cones, subdifferentials and coderivatives.
Naturally, these new rules exhibit a number of similarities with their classical counterparts concerning both
the structure of resulting formulas as well as the associated qualification conditions.
On the other hand, the directional calculus has also some specific hurdles related to the computation
of images or pre-images of the given directions in the considered mappings.
As mentioned in the Introduction, they lead either to additional terms in the respective formulas
or to additional qualification conditions. Expectantly, most qualification conditions have a directional character,
which follows from the fact that in this development one needs a ``regular'' behavior of feasibility mappings
only in the relevant directions. In the rules relying (partially) on Theorem \ref{The : ConstSetForw}
we make use of a special ``inner calmness'' property which arose (under a different name)
also in completely different contexts. In Section 6 we collected some representative (classes of) problems,
where directional limiting objects are helpful and the results of this paper enable the user to compute them.

\section*{Acknowledgments}
The research of the first two authors was supported by the Austrian Science
Fund (FWF) under grant P29190-N32. The research of the third author was supported by
the Grant Agency of the Czech Republic, projects 17-04301S and 17-08182S and the Australian Research Council, project DP160100854F.

\end{document}